\documentclass[11pt]{article}
\usepackage[a4paper,centering,vscale=0.75,hscale=0.7]{geometry}

\usepackage{mathtools}
\usepackage{amsthm,amssymb}
\usepackage{mathrsfs}
\usepackage{color}
\usepackage{textcomp}
\usepackage[implicit=false]{hyperref}

\newtheorem{thm}{Theorem}[section]
\newtheorem{lemma}[thm]{Lemma}
\newtheorem{prop}[thm]{Proposition}
\newtheorem{coroll}[thm]{Corollary}
\newtheorem{defi}[thm]{Definition}
\theoremstyle{remark}
\newtheorem{rmk}[thm]{Remark}
\newtheorem{ex}[thm]{Example}

\renewcommand{\div}{\operatorname{div}}
\newcommand{\dom}{\mathsf{D}}
\newcommand{\E}{\mathop{{}\mathbb{E}}}
\newcommand{\cE}{\mathscr{E}}
\newcommand{\cF}{\mathscr{F}}
\newcommand{\cL}{\mathscr{L}}
\renewcommand{\P}{\mathbb{P}}
\newcommand{\erre}{\mathbb{R}}
\newcommand{\enne}{\mathbb{N}}

\newcommand{\embed}{\hookrightarrow}
\newcommand{\wto}{\rightharpoonup}

\DeclarePairedDelimiter\abs{\lvert}{\rvert}
\DeclarePairedDelimiter\norm{\lVert}{\rVert}
\DeclarePairedDelimiterX\ip[2]{\langle}{\rangle}{#1,#2}

\numberwithin{equation}{section}

\newif\ifbozza
\bozzafalse

\def\luca #1{{\color{red}#1}}
\def\luca #1{#1}

\title{A variational approach to dissipative
    SPDEs\\ with singular drift}
  \author{Carlo Marinelli\thanks{Department of Mathematics, University
      College London, Gower Street, London WC1E 6BT, United
      Kingdom. URL: \texttt{http://goo.gl/4GKJP}} 
  \and Luca Scarpa\thanks{Department of Mathematics, University
      College London, Gower Street, London WC1E 6BT, United
      Kingdom. E-mail: \texttt{luca.scarpa.15@ucl.ac.uk}}}
\date{March 24, 2017}

\begin{document}
\maketitle

\begin{abstract}
  We prove global well-posedness for a class of dissipative semilinear
  stochastic evolution equations with singular drift and
  multiplicative Wiener noise. In particular, the nonlinear term in
  the drift is the superposition operator associated to a maximal
  monotone graph everywhere defined on the real line, on which neither
  continuity nor growth assumptions are imposed. The hypotheses on the
  diffusion coefficient are also very general, in the sense that the
  noise does not need to take values in spaces of continuous, or
  bounded, functions in space and time. Our approach combines
  variational techniques with a priori estimates, both
  pathwise and in expectation, on solutions to regularized equations.
\end{abstract}


\section{Introduction}
\label{sec:intro}
Our aim is to establish existence and uniqueness of solutions, and
their continuous dependence on the initial datum, to the following
semilinear stochastic evolution equation on $L^2(D)$, with
$D \subset \erre^n$ a bounded domain:
\begin{equation}
  \label{eq:0}
  dX(t) + AX(t)\,dt + \beta(X(t))\,dt \ni B(t,X(t))\,dW(t),
  \qquad X(0)=X_0,
\end{equation}
where $A$ is a linear maximal monotone operator on $L^2(D)$ associated
to a coercive Markovian bilinear form, $\beta$ is a maximal monotone
graph in $\erre\times\erre$ defined everywhere, $W$ is a cylindrical
Wiener process on a separable Hilbert space $U$, and $B$ takes values
in the space of Hilbert-Schmidt operators from $U$ to $L^2(D)$ and
satisfies suitable Lipschitz continuity assumptions. Precise
assumptions on the data of the problem and on the definition of
solution are given in Section \ref{sec:results} below.  Since any
increasing function $\beta_0:\erre\to\erre$ can be extended in a
canonical way to a maximal monotone graph of $\erre\times\erre$ by
``filling the gaps'' (i.e., setting
$\beta(x):=[\beta_0(x^-),\beta_0(x^+)]$ for all $x\in\erre$, where
$\beta(x^-)$ and $\beta(x^+)$ denote the limit from the left and from
the right of $\beta_0$ at $x$, respectively), Equation \eqref{eq:0}
can be interpreted as a formulation of the stochastic evolution
equation
\[
  dX(t) + AX(t)\,dt + \beta_0(X(t))\,dt = B(t,X(t))\,dW(t),
  \qquad X(0)=X_0.
\]
Semilinear equations with singular and rapidly growing drift appear,
for instance, in mathematical models of Euclidean quantum field theory
(see, e.g., \cite{Kawa} for an equation with exponentially growing
drift), and, most importantly for us, cannot be directly treated with
the existing methods, hence are interesting from a purely mathematical
perspective as well. In particular, the variational approach (see
\cite{KR-spde,Pard}) works only assuming that $\beta$ satisfies
suitable polynomial growth conditions depending on the dimension $n$
of the underlying Euclidean space (see also \cite[pp.~137-ff.]{LiuRo}
for improved sufficient conditions, still dependent on the dimension),
whereas most available results relying on the semigroup approach
require just polynomial growth, although usually compensated by rather
stringent hypotheses on the noise (see, e.g.,
\cite{cerrai03,DP-K}). Under natural assumptions on the noise,
well-posedness in $L^p$ spaces is proven, with different methods, in
\cite{KvN2}, under the further assumption that $\beta$ is locally
Lipschitz continuous, and in \cite{cm:rd}. A common basis for both
works is the semigroup approach on UMD Banach spaces.  A special
mention deserves the short note \cite{Barbu:Lincei}, where the author
considers problem \eqref{eq:0} with $A=-\Delta$ and $B$ independent of
$X$, and proves existence of a pathwise solution\footnote{To avoid
  misunderstandings, we should clarify once and for all that with this
  expression we do \emph{not} refer to a solution in the sense of
  rough paths, but simply ``with $\omega$ fixed''.} assuming that the
solution $Z$ to the equation with $\beta \equiv 0$ (i.e., the
stochastic convolution) is jointly continuous in space and
time. Furthermore, assuming that
\[
\E\int_0^T\!\!\int_D j(Z) < \infty,
\]
where $j$ is a primitive of $\beta$, he obtains that the pathwise
solution may admit a version that can be considered as a generalized
mild solution to \eqref{eq:0}. This is the only result we are aware of
about existence of solutions to stochastic semilinear parabolic
equations \emph{without} growth assumptions on the drift in any
dimension. 
\luca{%
It is well known that a well-posedness theory for stochastic evolution
equations on a Hilbert space $H$ of the type
\begin{equation*}
  du + Au\,dt \ni B(u)\,dW, \qquad u(0)=u_0,
\end{equation*}
with $A$ an arbitrary (nonlinear) maximal monotone operator, is, in
full generality, not yet available, even if $B$ does not depend on $u$
and is a fixed non-random operator. However, a satisfactory treatment
in the finite-dimensional case has been given by Pardoux and
R\u{a}\c{s}canu in \cite[{\S}4.2]{pardoux-rascanu}, where the authors
consider stochastic differential equations in $\erre^n$ of the type
\[ 
dX_t + A(X_t)\,dt + F(t,X_t)\,dt \ni G(t, X_t)\,dB_t,
\]
where $A$ is a (multivalued) maximal monotone operator whose domain
has non-emtpy interior, $B$ is a $k$-dimensional Wiener process, $G$
satisfies standard Lipschitz continuity assumptions, and $F(t,\cdot)$
is continuous and monotone (not necessarily Lipschitz
continuous). While the assumptions on $A$ are not restrictive in
finite dimensions, unbounded linear operators generating contraction
semigroups in infinite-dimensional spaces, as in our case, have dense
domain, whose interior is hence empty.}

On the other hand, in the deterministic setting complete results have
long been known for equations of the type
\[
\frac{du}{dt} + Au \ni f, \qquad u(0)=u_0,
\]
even in the much more general setting where $A$ is a (multivalued)
$m$-accretive operator on a Banach space $E$ and $f \in L^1(0,T;E)$
(see, e.g., \cite{barbu, Bmax}).
Although a solution to the general stochastic problem does not
currently seem within reach, significant results have been obtained in
special cases: apart of the above-mentioned works on semilinear
equations, well-posedness for the stochastic porous media equation
under fairly general assumptions is known (see \cite{BDPR-porous},
where the same hypotheses on $\beta$ imposed here are used and the
noise is assumed to satisfy suitable boundedness conditions, and
\cite{cm:IDAQP09} for an extension to jump noise). Moreover, the
variational theory by Pardoux, Krylov and Rozovski\u{\i} is
essentially as complete as the corresponding deterministic theory. As
mentioned above, however, large classes of maximal monotone operators
on $H=L^2(D)$ cannot be cast in the variational framework.

The main contribution of this work is a well-posedness result for
\eqref{eq:0} under the most general conditions known so far, to the
best of our knowledge. These conditions are quite sharp for $A$, but
not for $\beta$. In particular, the conditions on $A$ are close to
those needed to show that $A+\beta(\cdot)$ is maximal monotone on
$L^2(D)$, but the hypothesis that $\beta$ is finite on the whole real
line is not needed in the deterministic theory. Finally, the
conditions on $B$ are the natural ones to have function-valued noise,
and are in this sense as general as possible. Equations with white
noise in space and time, that have received much attention lately, are
not within the scope of our approach (nor of others, most likely,
under such general conditions on $\beta$).

In forthcoming work we shall extend our well-posedness results to
equations where $A$ is a nonlinear operator satisfying suitable
Leray-Lions conditions (thus including the $p$-Laplacian, for instance),
as well as to equations driven by discontinuous noise.

Let us now briefly outline the structure of the paper and the main
ideas of the proof. Section \ref{sec:results} contains the statement
of the main well-posedness result, and in Section \ref{sec:ex} we
discuss the hypotheses on the drift and diffusion coefficients,
providing corresponding examples. After collecting useful
preliminaries in Section \ref{sec:prelim}, we consider in Section
\ref{sec:reg} a version of equation \eqref{eq:0} with additive noise
satisfying a strong boundedness assumption. Using the Yosida
regularization of $\beta$, we obtain a family of approximating
equations with Lipschitz coefficients, which can be treated by the
standard variational theory. The solutions to such equations are shown
to satisfy suitable uniform estimates, both pathwise and in
expectation. Such estimates allow us to obtain key regularity and
integrability properties for the solution to the equation with
additive bounded noise. A crucial role is played by Simon's
compactness criterion, which is applied pathwise, and by compactness
criteria in $L^1$ spaces, applied both pathwise and in expectation. It
is, in essence, precisely this interplay between pathwise and
``averaged'' arguments that permits to avoid many restrictive
hypotheses of the existing literature. An abstract version of Jensen's
inequality for positive operators, combined with the lower
semicontinuity of convex integrals, is also an essential tool. In
Section \ref{sec:add} we prove well-posedness for equations with
additive noise removing the boundedness assumption of the previous
section. This is accomplished by a further regularization scheme, this
time on the diffusion operator $B$, and by a priori estimates for
solutions to the regularized equations. A key role is played again by
a combination of estimates and passages to the limit both pathwise and
in expectation. We also prove continuity of the solution map with
respect to the initial datum and the diffusion coefficient, by means
of It\^o's formula and regularizations, for which smoothing properties
of the resolvent of $A$ are essential.  Finally, in Section
\ref{sec:pf} we obtain well-posedness in the general case by a
fixed-point argument, using the Lipschitz continuity of $B$
only. Introducing weighted spaces of stochastic processes, we obtain
directly global well-posedness, thus avoiding a tedious construction
by ``patching'' local solutions.

Some tools and reasonings used in this work are obviously not new:
weak compactness arguments in $L^1$, for instance, are extensively
used in the literature on partial differential equations (see, e.g.,
\cite{BoCro, Bre-mm} and references therein), as well as, to a lesser
extent, in the stochastic setting (cf.~\cite{Barbu:Lincei,
  BDPR-porous, cm:SIMA12}).  However, even where similarities are
present, our arguments are considerably streamlined and more
general. The pathwise application of Simon's compactness criterion,
made possible by a construction based on the variational framework,
seems to be new, at least in the context of stochastic evolution
equations. It is in fact somewhat surprising that the
variational setting, which notoriously fails when dealing with
semilinear equations, is at a basis of an approach that leads to
well-posedness of those same equations, even with singular and rapidly
increasing drift.

\medskip

\noindent
\textbf{Acknowledgments.} The authors are partially supported by a
grant of The Royal Society. The first-named author is very grateful to
Prof.~S.~Albeverio for the warm hospitality and the excellent working
conditions at the Interdisziplin\"ares Zentrum f\"ur Komplexe Systeme,
University of Bonn, where parts of this work were written. Two
anonymous referees provided useful comments and suggestions that led
to a better presentation of our results.


\ifbozza\newpage\else\fi
\section{Main result}
\label{sec:results}
In this section, after fixing notation and conventions used throughout
the paper, we state our main result.

\subsection{Notation}
All functional spaces will be defined on a smooth bounded domain
$D \subset \erre^n$. We shall denote $L^2(D)$ by $H$ and its inner
product by $\ip{\cdot}{\cdot}$. The domain and the range of a generic
map $G$ will be denoted by $\dom(G)$ and $\mathsf{R}(G)$,
respectively. If $E$ and $F$ are subsets of a topological space, we
shall write $E \embed F$ to mean that $E$ is continuously embedded in
$F$, i.e. that $E$ is a subset of $F$ and that the injection
$i:E \to F$ is continuous. Let $E$, $F$ be Banach spaces.  \luca{The
  space of linear continuous operators from $E$ to $F$ is denoted by
  $\cL(E,F)$ if endowed with the operator norm, and by $\cL_s(E,F)$ if
  endowed with the strong operator topology, that is, $T_n \to T$ in
  $\cL_s(E,F)$ if $T_nu \to Tu$ in $F$ for all $u \in E$.  If
  $F=\erre$, $\cL(E,\erre)$ is the dual space $E^*$.}  If $E$ and $F$
are Hilbert spaces, we shall denote the space of Hilbert-Schmidt
operators from $E$ to $F$ by $\cL^2(E,F)$.

We shall occasionally use the symbols $\wto$ and
$\xrightharpoonup{*}$ to denote convergence in the weak and weak*
topology of Banach spaces, respectively, while the symbol $\to$ is
reserved for convergence in the norm topology.

All random quantities will be defined on a fixed probability space
$(\Omega,\cF,\P)$ endowed with a right-continuous and saturated
filtration $\mathbb{F}:=(\cF_t)_{t\in[0,T]}$, where $T$ is a positive
number. All expressions involving random quantities are meant to hold
$\P$-almost surely, unless otherwise stated. With $W$ we shall denote
a cylindrical Wiener process on a separable Hilbert space $U$, that
may coincide with $H$, but does not have to. We shall use the standard
notation of stochastic calculus, such as $K \cdot W$ to mean the
stochastic integral of $K$ with respect to $W$, and, for a process $X$
taking values in a normed space $E$,
$X^*_t := \operatorname{ess\,sup}_{s\in[0,t]} \norm{X(s)}_E$.

Let $E$ be a separable Banach space. Given a measure space
$(Y,\mathscr{A},\mu)$ and $p \in [1,\infty]$, we shall denote the
space of strongly measurable functions from $\phi:Y \to E$ such that
$\norm{\phi}_E \in L^p(Y)$ by $L^p(Y;E)$. Moreover, we shall write
$L^2(\Omega;L^\infty(0,T;E))$ to denote the space of $\mathscr{F}
\otimes \mathscr{B}([0,T])$-measurable processes $\phi:\Omega \times
[0,T] \to E$ such that
\[
\norm[\big]{\phi}_{L^2(\Omega;L^\infty(0,T;E))} :=
\Bigl( \E \operatorname*{ess\,sup}_{t\in[0,T]} \norm{\phi(t)}^2_E
\Bigr)^{1/2} < \infty.
\]
Given an interval $I \subseteq \erre$, the space of continuous and of
weakly continuous functions from $I$ to $E$ will be denoted by
$C(I;E)$ and $C_w(I;E)$, respectively.

We shall write $a \lesssim b$ to mean that there exists a constant $N$
such that $a \leq Nb$. If such a constant depends on certain
parameters of interest, we shall put these in parentheses or write
them as subscripts.

\subsection{Assumptions}
The following assumptions on the data of the problem are assumed to be
in force throughout and will not always be recalled explicitly.
\smallskip\par\noindent
\textbf{Assumption A.} Let $V$ be Hilbert space that is densely,
continuously, and compactly embedded in $H$. The linear operator $A$
belongs to $\cL(V,V^*)$ and satisfies the following properties:
\begin{itemize}
\item[(i)] there exists $C>0$ such that
  \[
  \ip{Av}{v} \geq C \norm{v}_V^2 \qquad \forall v \in V;
  \]
\item[(ii)] the part of $A$ in $H$ admits a unique $m$-accretive
  extension $A_1$ in $L^1(D)$;
\item[(iii)] the resolvent
  $\bigl((I+\lambda A_1)^{-1}\bigr)_{\lambda>0}$ is sub-Markovian;
\item[(iv)] there exists $m \in \mathbb{N}$ such that
  \[
  \norm[\big]{(I+A_1)^{-m}}_{\cL(L^1(D),L^\infty(D))} < \infty.
  \]
\end{itemize}
\smallskip\par\noindent
Here we have used $\ip{\cdot}{\cdot}$ also to denote the duality
pairing of $V$ and $V^*$, which is compatible with the scalar product
in $H$. In fact, identifying $H$ with its dual, one has the so-called
Gel{\textquotesingle}fand triple
\[
V \embed H \embed V^*,
\]
where both embeddings are dense (see, e.g.,
\cite[{\S}2.9]{LiMa1}). 
\luca{Moreover, we recall that the part of $A$ in
$H$ is the operator
$A_2$ on $H$ defined as
$\dom(A_2):=\{x\in V: Au\in H\}$ and
$A_2x:= Ax$ for all $x\in \dom(A_2)$.
If one identifies the operators with their graphs,
this is equivalent to setting $A_2:=A\cap(V\times H)$.}
We shall often refer
to condition (i) as the coercivity of $A$. The sub-Markovianity
condition (iii) amounts to saying that, for all functions
$f \in L^1(D)$ such that $0 \leq f \leq 1$, one has
\[
0 \leq (I+A_1)^{-1}f \leq 1.
\]
In other words, $(I+A_1)^{-1}$ is positivity preserving and contracting
in $L^\infty(D)$.

From Section \ref{sec:reg} onwards, we shall often use the symbol $A$
to denote also $A_1$ and $A_2$.

Let us observe that if $A$ is the negative Laplacian with
Dirichlet boundary conditions, all hypotheses are met. Much wider
classes of operators satisfying hypotheses (i)-(iv) will be given below.

\medskip

\noindent\textbf{Assumption B.} $\beta$ is a maximal monotone graph of
$\erre \times \erre$ such that $\dom(\beta)=\erre$, $0 \in \beta(0)$,
and its potential $j$ is even.
\smallskip\par\noindent
We recall that the potential $j$ of $\beta$ is the convex, proper,
lower semicontinuous function $j: \erre \to \erre_+$, with $j(0)=0$,
such that $\partial j = \beta$, where $\partial$ stands for the
subdifferential in the sense of convex analysis.\footnote{See
  {\S}\ref{ssec:cvx} below for a summary of the notions of convex
  analysis and of the theory of nonlinear monotone operators used
  throughout.}

\medskip

\noindent\textbf{Assumption C.} The diffusion coefficient
\[
B:\Omega \times[0,T] \times H \to \cL^2(U,H)
\]
is Lipschitz continuous and grows linearly in its third argument,
uniformly over $\Omega \times [0,T]$, i.e., there exist constants
$L_B$, $N_B$ such that
\begin{align*}
\norm[\big]{B(\omega,t,x) - B(\omega,t,y)}_{\cL^2(U,H)} 
&\leq L_B \norm{x-y}_H,\\
\norm[\big]{B(\omega,t,x)}_{\cL^2(U,H)}
&\leq N_B \bigl(1 + \norm{x}_H \bigr)
\end{align*}
for all $\omega \in \Omega$, $t \in [0,T]$, and $x$, $y \in H$.
Moreover, $B(\cdot,\cdot,x)$ is progressively measurable for all
$x \in H$, i.e., for all $t \in [0,T]$, the map
$(\omega,s) \mapsto B(\omega,s,x)$ from $\Omega \times [0,t]$, endowed
with the $\sigma$-algebra $\cF_t \otimes \mathscr{B}([0,t])$, to
$\cL^2(U,H)$, endowed with its Borel $\sigma$-algebra, is strongly
measurable. We recall that, since $U$ and $H$ are separable, the space
of Hilbert-Schmidt operators $\cL^2(U,H)$ is itself a separable
Hilbert space, hence strong and weak measurability coincide. Whenever
we deal with maps with values in separable Banach spaces, since strong
and weak measurability coincide, we shall drop the qualifier
``strong''.

\subsection{The well-posedness result}
\begin{defi}
  Let $X_0$ be an $H$-valued $\cF_0$-measurable random variable. A
  \emph{strong solution} to the stochastic equation \eqref{eq:0} is a
  \luca{pair} $(X,\xi)$ satisfying the following properties:
  \begin{itemize}
  \item[\emph{(i)}] $X$ is a measurable adapted $V$-valued
    process such that $AX \in L^1(0,T;V^*)$ and
    $B(\cdot,X) \in L^2(0,T;\cL^2(U,H))$;
  \item[\emph{(ii)}] $\xi$ is a measurable adapted
    $L^1(D)$-valued process such that $\xi \in L^1(0,T;L^1(D))$ and
    $\xi \in \beta(X)$ almost everywhere in $(0,T) \times D$;
  \item[\emph{(iii)}] one has, as an equality in $L^1(D) \cap V^*$,
   \[
   X(t)+\int_0^t{AX(s)\,ds}+\int_0^t{\xi(s)\,ds} = X_0
   +\int_0^t B(s,X(s))\,dW(s)
    \]
   for all $t \in [0,T]$.
  \end{itemize}
\end{defi}
Note that $L^1(D) \cap V^*$ is not empty because $D$ has
finite Lebesgue measure, hence, for instance, $H$ is contained in both
spaces.

\smallskip

Let us denote by $\mathscr{J}$ the set of pairs $(\phi,\zeta)$,
where $\phi$ and $\zeta$ are measurable adapted processes
with values in $H$ and $L^1(D)$, respectively, such that
\begin{align*}
\phi &\in L^2(\Omega;L^\infty(0,T;H)) \cap L^2(\Omega;L^2(0,T;V)),\\
\zeta &\in L^1(\Omega \times [0,T] \times D),\\
j(\phi) + j^*(\zeta) &\in L^1(\Omega \times [0,T] \times D).
\end{align*}
We shall say that \eqref{eq:0} is well posed in $\mathscr{J}$ \luca{if there
exists a unique process in $\mathscr{J}$ which is a strong solution and} such that
the solution map $X_0 \mapsto X$ is continuous from $L^2(\Omega;H)$ to
$L^2(\Omega;L^\infty(0,T;H)) \cap L^2(\Omega;L^2(0,T;V))$.

\smallskip

The central result of this work is the following.
\begin{thm}  \label{th:main}
  Let $X_0 \in L^2(\Omega,\cF_0,\P;H)$. Then \eqref{eq:0} is well-posed
  in $\mathscr{J}$. \luca{Moreover, 
  the solution map $X_0 \mapsto X$ is Lipschitz continuous and the
  paths of $X$ are weakly continuous with values in $H$.
  }
\end{thm}
\smallskip

\luca{Let us stress the fact that} the more general problem of unconditional well-posedness
(i.e. without the extra condition that strong solutions belong to
$\mathscr{J}$) remains open and is beyond the scope of the techniques
used in this work. In particular, we can only prove uniqueness of
solutions within $\mathscr{J}$.

\section{Examples and remarks}
\label{sec:ex}
Some comments and examples on the assumptions on the data of the
problem are in order. In particular, the hypotheses on $A$ deserve
special attention. The coercivity condition
$\ip{Av}{v} \geq C\norm{v}_V^2$ for all $v \in V$ is equivalent to
$A \in \cL(V,V^*)$ being determined by a bounded
$V$-elliptic\footnote{We prefer this terminology, taken from
  \cite{LiMa1}, over the currently more common ``$V$-coercive'', to
  avoid possible confusion with related terminology used in the theory
  of Dirichlet forms, where coercivity is meant in a somewhat
  different sense (cf.~\cite[Definition~2.4, p.~16]{MR:DF}).}
bilinear form $\cE:V \times V \to \erre$, i.e. such that
\[
\abs{\cE(u,v)} \lesssim \norm{u}_V \norm{v}_V, \qquad
\cE(v,v) \geq C \norm{v}_V^2 \qquad \forall u,v \in V.
\]
This is an immediate consequence of the Lax-Milgram theorem, which
also implies that $A$ is an isomorphism between $V$ and $V^*$ (see,
e.g., \cite[{\S}5.2]{ISEM18} or \cite[Lemma~1.3]{Ouhabaz}).

The bilinear form $\cE$ can also be seen as a closed unbounded form on
$H$ with domain $V$. This defines a (unique) linear $m$-accretive
operator $A_2$ on $H$, that is nothing else than the part of $A$ in
$H$ (see, e.g., \cite[{\S}5.3]{ISEM18} or \cite[p.~34]{Ouhabaz}).
Conversely, given a positive closed bilinear form $\cE$ on $H$ with
dense domain $\dom(\cE)$ satisfying the \emph{strong sector
  condition}\footnote{Throughout this section we shall follow the
  terminology on Dirichlet forms of \cite{MR:DF}.}
\[
\abs{\cE(u,v)} \lesssim \cE(u,u)^{1/2} \cE(v,v)^{1/2} 
\qquad \forall u, v \in \dom(\cE),
\]
and such that $\cE(u,u) > 0$ for all $u \in \dom(\cE)$, $u \neq 0$, 
setting $V:=\dom(\cE)$ with inner product given by the symmetric part
$\cE^s$ of $\cE$, that is
\[
\cE^s(u,v) := \frac12 \bigl( \cE(u,v) + \cE(v,u) \bigr), 
\qquad u, v \in \dom(\cE),
\]
there is a unique linear operator $A \in \cL(V,V^*)$ such that
$\cE(u,v)=\ip{Au}{v}$ for all $u, v \in V$. This amounts to trivial
verifications, since, obviously, $\cE(u,u)=\cE^s(u,u)$ for all
$u \in \dom(\cE)$.
As a particular case, let $A'$ be a linear positive self-adjoint
(unbounded) operator $H$ such that $\ip{A'u}{u}>0$ for all $u \in
\dom(A)$, $u \neq 0$. Then $A'$ admits a square root $\sqrt{A'}$,
which is in turn a linear positive self-adjoint operator on $H$. One
can then define the Hilbert space $V:=\dom(\sqrt{A'})$, endowed with
the inner product
\[
\ip{u}{v}_V := \ip[\big]{\sqrt{A'}u}{\sqrt{A'}v},
\]
and the symmetric bounded bilinear form $\cE:V \times V \to \erre$,
\[
\cE(u,v) := \ip[\big]{\sqrt{A'}u}{\sqrt{A'}v}, \qquad u,v \in V,
\]
which is obviously $V$-elliptic. By a theorem of Kato
(\cite[Theorem~2.23, p.~331]{Kato}), there is in fact a bijective
correspondence between linear positive self-adjoint operators on $H$
and positive densely-defined closed symmetric bilinear forms.
More generally, if $A'$ is a linear (unbounded) $m$-accretive operator
on $H$ such that
\[
\abs[\big]{\ip{A'u}{v}} \lesssim \ip{A'u}{u}^{1/2} \ip{A'v}{v}^{1/2}
\qquad \forall u, v \in \dom(A'),
\]
and $\ip{A'u}{u}>0$ for all $u \in \dom(A')$, $u \neq 0$, then there
exists a (unique) closed $V$-elliptic bilinear form $\cE$ that
determines an operator $A \in \cL(V,V^*)$, with $V:=\dom(\cE)$ and
$\ip{\cdot}{\cdot}_V:=\cE^s$, such that $A'$ is the part on $H$ of
$A$. This follows, for instance, by \cite[p.~27]{MR:DF}.

Note, however, that in the previous examples $V$ may \emph{not} be
continuously embedded in $H$, unless $\cE$ satisfies a Poincar\'e
inequality, i.e. $\norm{u}_H^2 \lesssim \cE(u,u)$ for all
$u \in \dom(\cE)$ (as is the case, for instance, for the Dirichlet
Laplacian). This limitation is resolved by the following important
observation: all our well-posedness result continues to hold if we
assume, in place of hypothesis (i), the following weaker one:
\begin{itemize}
\item[(i')] there exist constants $C_1>0$, $C_2 \in \erre$ such that
\[
\ip{Av}{v} \geq C_1 \norm{v}_V^2 - C_2 \norm{v}_H^2
\qquad \forall v \in V,
\]
\end{itemize}
which is clearly equivalent to assuming that $\tilde{A}:=A+C_2I$ is
$V$-elliptic. Under this assumption, equation \eqref{eq:0} can
equivalently be written as
\[
dX(t) + \tilde{A}X(t)\,dt + \beta(X(t))\,dt = C_2X(t)\,dt +
B(t,X(t))\,dW(t).
\]
The only added complication in the proofs to follow would be the
appearance of functional spaces with an exponential weight in time,
very much as in the proof of Proposition \ref{prop:contra} below. An
analogous argument, in a slightly different context, is developed in
detail in \cite{cm:rd}. This seemingly trivial observation allows to
considerably extend the class of operators $A$ that can be
treated. For instance, one has the following criterion.
\begin{lemma}
  A coercive closed form $\cE$ on $H$ uniquely determines an
  operator $A$ satisfying \emph{(i')}.
\end{lemma}
\begin{proof}
  The hypothesis of the Lemma means that $\cE$ is a densely defined
  bilinear form such that its symmetric part $\cE^s$ is closed and
  $\cE$ satisfies the \emph{weak sector condition}
  \[
  \abs[\big]{\cE_1(u,v)} \lesssim \cE_1(u,u)^{1/2} \cE_1(v,v)^{1/2}
  \qquad \forall u,v \in \dom(\cE),
  \]
  where $\cE_1:=\cE + I$. In other words, $\cE$ satisfies the weak
  sector condition if the shifted form $\cE + I$ satisfies the strong
  sector condition. Therefore, adapting in the obvious way an argument
  used above, it is enough to take $V:=\dom(\cE)$ with inner product
  $\ip{\cdot}{\cdot}_V:=\ip{\cdot}{\cdot}_H + \cE^s$ to obtain that
  the generator $A_2$ of $\cE$ can be (uniquely) extended to an
  operator $A \in \cL(V,V^*)$ satisfying (i') with $C_1=C_2=1$.
\end{proof}
Note that in all the above constructions one has $V \embed H$ densely
and continuously (under appropriate assumptions), but the embedding is
\emph{not} necessarily compact. The latter condition has to be proved
depending on the situation at hand. For a general compactness
criterion in terms of ultracontractivity properties, see Proposition
\ref{prop:cpt} below.

\smallskip

As regards condition (ii), the simplest sufficient condition ensuring
that $A_2$ admits an $m$-accretive extension $A_1$ in $L^1(D)$ is that
$-A_2$ is the generator of a symmetric Markovian semigroup of
contractions $S_2$ on $H$, or, equivalently, that $A_2$ is positive
self-adjoint with a Markovian resolvent. In fact, this implies that,
for any $p \in \mathopen[1,\infty\mathclose[$, there exists a (unique)
symmetric Markovian semigroup of contractions $S_p$ on $L^p(D)$ such
that all $S_p$, $1 \leq p < \infty$, are consistent, hence the
corresponding negative generators $A_p$ coincide on the intersections
of their domains (see, e.g., \cite[Theorem~1.4.1]{Dav-HK}). In the
general case, i.e. if $A_2$ is not self-adjoint, the same conclusion
remains true if the semigroup $S_2$ and its adjoint $S_2^*$ are both
sub-Markovian, or, equivalently, if $S_2$ is sub-Markovian and
$L^1$-contracting (cf.~ \cite[Lemma~10.13 and Theorem~10.15]{ISEM18}
or \cite[Corollary~2.16]{Ouhabaz}). In particular, if $A_2$ is the
generator of a Dirichlet form on $H$, these conclusions
hold. Moreover, since the resolvent of $A_1$ is sub-Markovian if and
only if the resolvent of $A_2$ is sub-Markovian, we obtain the
following complement to the previous Lemma.
\begin{lemma}
 A Dirichlet form $\cE$ on $H$ uniquely determines an operator $A$ satisfying
 \emph{(i')}, \emph{(ii)}, and \emph{(iii)}.
\end{lemma}
Without assuming that $S_2^*$ is sub-Markovian (which is the case, for
instance, if $A$ is determined by a semi-Dirichlet form on $H$, so
that (i') and (iii) only are satisfied), we note that $D(A_2)$ is dense in
$L^1(D)$, and the image of $I+A_2$ is dense in $L^1(D)$: the former
assertion follows by $D(A_2) \subset L^2(D)$ densely and
$L^2(D) \subset L^1(D)$ densely and continuously. Moreover, since
$A_2$ generates a contraction semigroup in $L^2(D)$, the
Lumer-Phillips theorem (see, e.g., \cite[p.~83]{EnNa}) implies that
$\mathsf{R}(I+A_2)=L^2(D)$, hence $\mathsf{R}(I+A_2)$ is dense in
$L^1(D)$. The Lumer-Phillips theorem again guarantees that the closure
of $A_2$ in $L^1(D)$ is $m$-accretive if $A_2$ is accretive in
$L^1(D)$. The latter property is often not difficult to verify in
concrete examples.

\smallskip

The most delicate condition is (iv), i.e. the ultracontractivity of
suitable powers of the resolvent of $A_1$. If $A_2$ is self-adjoint, a
simple duality arguments shows that, for any $t \geq 0$,
\[
\norm[\big]{S_2(t)}_{\cL(L^1,L^\infty)} \leq 
\norm[\big]{S_2(t/2)}^2_{\cL(L^2,L^\infty)}.
\]
Sufficient conditions for $S_2(t)$ to be bounded from $L^2(D)$ to
$L^\infty(D)$ are known in terms, for instance, of logarithmic Sobolev
inequalities, Sobolev inequalities, and Nash inequalities (see, e.g.,
\cite[Chapter ~2]{Dav-HK} and \cite[Chapter~6]{Ouhabaz}). The
non-symmetric case is more difficult, but ultracontractivity estimates
are known in many special cases, such as in the examples that we are
going to discuss next. Ultracontractivity estimates for powers of the
resolvent can then be obtained from estimates for the semigroup, as
explained below. The following result (probably known, but for which
we could not find a reference) shows that hypothesis (iv) guarantees
that the embedding $\dom(\cE) \embed H$ is compact, thus answering a
question left open above.
\begin{prop}  \label{prop:cpt}
  Let $A_2$ be the generator of a closed coercive form $\cE$ in
  $H$. If there exists $m \in \mathbb{N}$ such that the $m$-th power
  of the resolvent of $A_2$ is bounded from $L^2(D)$ to $L^\infty(D)$,
  then $\dom(\cE)$ is compactly embedded in $H$.
\end{prop}
\begin{proof}
  Let $(u_k)_k$ be a bounded sequence in $\dom(\cE)$, i.e., there exists a
  constant $N$ such that
  \[
  \norm{u_k}^2_H + \cE^s(u_k,u_k) < N \qquad \forall k \in \mathbb{N}.
  \]
  In particular, there exists a subsequence of $k$, denoted by the
  same symbol, such that $u_k$ converges weakly to $u$ in $H$ as
  $k \to \infty$. The goal is to show that the convergence is in fact
  strong. Since $\dom(A_2^m) \subset L^\infty(D)$ by assumption, it
  follows by a result of Arendt and Bukhvalov, see
  \cite[Theorem~4.16(b)]{AreBukh}, that the resolvent
  $J_\lambda:=(I+\lambda A_2)^{-1}$ is a compact operator on $H$ for
  all $\lambda>0$. The triangle inequality yields
  \[
  \norm{u_k - u} \leq \norm{u_k - J_\lambda u_k} 
  + \norm{J_\lambda u_k - J_\lambda u}
  + \norm{J_\lambda u - u},
  \]
  where the second term on the right-hand side converges to zero as
  $k \to \infty$ by compactness of $J_\lambda$. Moreover, since
  $J_\lambda \to I$ in $\cL_s(H,H)$ as $\lambda \to 0$, the third term
  on the right-hand side can be made arbitrarily small. Therefore we
  only have to bound the first term on the right-hand side: note that
  $I - J_\lambda = \lambda A_\lambda$, where $A_\lambda$, $\lambda>0$,
  stands for the Yosida approximation of $A_2$, hence
  $\norm{u_k-J_\lambda u_k} = \lambda \norm{A_\lambda u_k}$, and
  \begin{align*}
  \ip{A_\lambda u_k}{u_k} &= \ip{A_\lambda u_k}{u_k-J_\lambda u_k+J_\lambda u_k}
  = \lambda \norm{A_\lambda u_k}^2 + \ip{A_\lambda u_k}{J_\lambda u_k}\\
  &\geq \lambda \norm{A_\lambda u_k}^2,
  \end{align*}
  where we have used, in the last step, the identity
  $A_\lambda=A_2J_\lambda$ and the monotonicity of $A_2$.
  Since, by \cite[Lemma~2.11(iii), p.~20]{MR:DF}, one has
  \[
  \abs[\big]{\cE_1^{(\lambda)}(u,v)} \lesssim \cE_1(u,u)^{1/2}
  \cE_1^{(\lambda)}(v,v)^{1/2} \qquad \forall u \in \dom(\cE), \, v \in H,
  \]
  where $\cE^{(\lambda)}(u,v):=\ip{A_\lambda u}{v}$, $u,v \in H$, and
  the implicit constant depends only on $\cE$, it follows that
  \[
  \cE_1^{(\lambda)}(u,u) \lesssim \cE_1(u,u) 
  \qquad \forall u \in \dom(\cE),
  \]
  hence
  \[
  \norm{u_k - J_\lambda u_k}^2 = \lambda^2 \norm{A_\lambda u_k}^2 
  \leq \lambda \ip{A_\lambda u_k}{u_k} = \lambda \cE_1^{(\lambda)}(u_k,u_k)
  \lesssim \lambda \cE_1(u_k,u_k).
  \]
  By the assumptions on the sequence $(u_k)$,
  \[
  \cE_1(u_k,u_k) = \norm{u_k}^2 + \cE(u_k,u_k) = \norm{u_k}^2 + \cE^s(u_k,u_k)
  \]
  is bounded uniformely over $k$, hence $\norm{u_k - J_\lambda u_k}^2$
  can be made arbitrarily small as well, thus proving the claim.
\end{proof}

\medskip

Let us now consider some concrete examples: we first consider the case
of $A$ being a suitable ``realization'' of a second-order
differential operator, and then of a nonlocal operator. 
\begin{ex}[Symmetric divergence-form operators]
  Consider the bilinear form $\cE$ on $V:=H^1_0(D)$ defined by
  \[
  \cE(u,v) := \ip[\big]{a\nabla u}{\nabla v} 
  = \luca{\sum_{j,k=1}^n} a_{jk} \partial_j u \partial_k v,
  \]
  where \luca{$a=(a_{jk})$ with $a_{jk} \in L^\infty(D)$} for all $j,k$, and
  $a_{jk}=a_{kj}$. The (formal) differential operator associated to
  $\cE$ is
  \[
  A_0u := -\div\bigl(a \nabla u\bigr), \qquad u \in C^\infty_c(D),
  \]
  where $C^\infty_c(D)$ stands for the set of infinitely
  differentiable functions with compact support contained in $D$.  The
  form $\cE$ is $V$-elliptic if there exists $C>0$ such that
  $\ip{a\xi}{\xi} \geq C\abs{\xi}^2$ for all \luca{$\xi \in \erre^n$.}
  Moreover, if there exists a positive function $\mu \in C(D)$ such
  that \luca{$\ip{a\xi}{\xi} \leq \mu(\xi) |\xi|^2$ for all
    $\xi\in D$,} then $A_2$ has sub-Markovian resolvent (details can
  be found, e.g., in \cite[Chapter~1]{Dav-HK} and, in much more
  generality, in \cite[Chapter~II]{MR:DF}). Ultracontractivity
  estimates follow as a special case of the corresponding estimates
  for non-symmetric forms treated next.
\end{ex}

\begin{ex}[Non-symmetric divergence-form operators with lower-order terms]
  Consider the differential operator on smooth functions
  \begin{align*}
  A_0u &:= -\div(a\nabla u) + b \cdot \nabla u - \div(cu) + a_0u\\
       &= - \sum_{j,k=1}^n \partial_j(a_{jk} \partial_k u)
          + \sum_{j=1}^n \bigl( b_j \partial_j u - \partial_j(c_ju)\bigr)
          + a_0u,
  \end{align*}
  where $a_{jk}$, $b_j$, $c_j$, $a_0 \in L^\infty(D)$, and the
  associated (non-symmetric) bilinear form $\cE$ on $V:=H^1_0(D)$
  is defined as
  \begin{align*}
  \cE(u,v) &= \ip{a\nabla u}{\nabla v} + \ip{b \cdot \nabla u}{v}
  + \ip{u}{c \cdot \nabla v} + \ip{a_0u}{v}\\
  &= \int_D \Bigl( \sum_{jk} a_{jk} \partial_j u \, \partial_k v
  + \sum_j \bigl( b_j \partial_ju \,v + c_ju \partial_jv\bigr)
  + a_0uv \Bigr).
  \end{align*}
  The bilinear form $\cE$ is continuous, as it easily follows from the
  boundedness of its coefficients. If there exists a constant $C>0$
  such that $\ip{a\xi}{\xi} \geq C\abs{\xi}^2$, then $\cE$ is not
  $V$-elliptic, but satisfies the weaker estimate
  \[
  \cE(u,u) \geq C_1\norm{u}_V^2 - C_2\norm{u}_H^2
  \qquad \forall u \in V,
  \]
  where $C_1>0$ and $C_2 \in \erre$ (see, e.g.,
  \cite[{\S}11.2]{ISEM18} or \cite[p.~100]{Ouhabaz}), i.e. the
  corresponding operator $A$ satisfies (i'), but not (i).  Using the
  Poincar\'e inequality, it is not difficult to show that $\cE$ is
  $V$-elliptic if the diameter of $D$ is small enough (see
  \cite[pp.~385--387]{DauLio:2}).  If we furthermore assume that
  $a_0 - \div c \geq0$ (in the sense of distributions), then the
  semigroup $S_2$ is sub-Markovian, and so is also the resolvent of
  $A_2$. Similarly, if $a_0 - \div b \geq 0$,\footnote{These two
    conditions involving $a_0$ and the divergence of $b$, $c$, are not
    restrictive, as they are close to necessary to ensure that the
    bilinear form $\cE$ is positive. This can be seen by a simple
    computation based on integration by parts,
    cf.~\cite[p.~48]{MR:DF}.}  then the semigroup $S_2$ is
  $L^1$-contracting (these results can be found, for instance, in
  \cite[Proposition~11.14]{ISEM18}, or deduced from
  \cite[{\S}4.3]{Ouhabaz}). As already mentioned above, this implies
  that $S_2$ can be extended to a consistent family of semigroups
  $S_p$ for all $p \in \mathopen[1,\infty\mathclose[$. Finally, let us
  discuss ultracontractivity: if $\cE$ is $V$-elliptic, and $S_2$ as
  well as $S_2^*$ are sub-Markovian, then a reasoning based on the
  Nash inequality
  \[
  \norm[\big]{u}_{L^2}^{2+4/n} \leq 
  N \norm[\big]{u}_{H^1_0}^{2} \norm[\big]{u}_{L^1}^{4/n}
  \qquad \forall u \in H^1_0,
  \]
  implies the estimate
  \[
  \norm[\big]{S_2(t)}_{\cL(L^1,L^\infty)} \leq N_1 t^{-n/2},
  \]
  where $N_1:=\bigl( Nn/(2\alpha) \bigr)^{n/2}$. For a proof, see,
  e.g., \cite[Theorem~12.3.2]{ISEM:HK} or \cite[p.~159]{Ouhabaz}.  The
  Laplace transform representation of the resolvent yields
  \[
  (I+\lambda A_1)^{-m} = 
  \frac{\lambda^m}{(m-1)!} \int_0^\infty t^{m-1} e^{-\lambda t} S(t)\,dt
  \]
  (see, e.g., \cite[p.~17]{ISEM18} or \cite[p.~21]{Pazy}), hence
  \[
  \norm[\big]{(I+\lambda A_1)^{-m}}_{\cL(L^1,L^\infty)} \lesssim
  \frac{\lambda^m}{(m-1)!} \int_0^\infty t^{m-1-n/2} e^{-\lambda t}\,dt.
  \]
  Thus it suffices to choose $m$ large enough to infer the
  ultracontractivity of the $m$-th power of the resolvent.
\end{ex}

\begin{ex}[Fractional Laplacian]
  Let $\Delta$ be the Dirichlet Laplacian on $H$. Since it is a
  positive self-adjoint operator, it follows that, for any
  $\alpha \in \mathopen]0,1\mathclose[$, $(-\Delta)^\alpha$ is itself
  a positive self-adjoint (densely defined) operator on
  $H$. Furthermore, the bilinear form
  \[
  \cE(u,v) := \ip[\big]{(-\Delta)^\alpha u}{v}
  = \ip[\big]{(-\Delta)^{\alpha/2}u}{(-\Delta)^{\alpha/2}v}, 
  \qquad u,v \in \dom\bigl((-\Delta)^{\alpha/2}\bigr),
  \]
  is a symmetric Dirichlet form on $H$, which, as already seen,
  uniquely determines an operator $A$ satisfying conditions (i'),
  (ii), and (iii): in particular,
  $V=\dom\bigl((-\Delta)^{\alpha/2}\bigr)$, equipped with the scalar
  product $\ip{\cdot}{\cdot}_V:=\ip{\cdot}{\cdot} + \cE$, and $A$ is
  just the extension of $(-\Delta)^\alpha$, generator of $\cE$, to
  $V$. In order to prove (iv), we are going to use again an argument
  based on the Nash inequality, which is however more involved as
  before. In particular, since $-\Delta$ satisfies the Nash inequality
  \[
  \norm[\big]{u}_{L^2}^{2+4/n} \lesssim 
  \ip[\big]{-\Delta u}{u} \norm[\big]{u}_{L^1}^{4/n}
  \qquad \forall u \in H^1_0,
  \]
  a result by Bendikov and Maheux, see \cite[Theorem~1.3]{Bend:Nash},
  implies that the fractional power $(-\Delta)^\alpha$ satisfies the
  Nash inequality
  \[
  \norm[\big]{u}_{L^2}^{2+4\alpha/n} \lesssim 
  \ip[\big]{(-\Delta)^\alpha u}{u} \norm[\big]{u}_{L^1}^{4\alpha/n}
  \qquad \forall u \in \dom(\cE).
  \]
  It follows by a general criterion of Varopoulos, Saloff-Coste and
  Coulhon (attributed to Ph.~B\'enilan), see
  \cite[Theorem~II.5.2]{VSC}, that the semigroup $S_\alpha$ on $H$
  generated by $(-\Delta)^\alpha$ satisfies the ultracontractivity
  estimate
  \[
  \norm[\big]{S_\alpha(t)}_{\cL(L^1,L^\infty)} \lesssim t^{-n/2\alpha},
  \]
  from which corresponding estimates for suitable powers of the
  resolvent can be deduced, as in the previous example.
\end{ex}
Related results on ultracontractivity and smoothing properties of
semigroups generated by non-local operators, arising as generators of
Markov processes, can be found, e.g., in \cite{Gentil:Levy,cm:AIHP14}.

\medskip

We proceed with a brief discussion about the relation between our
hypotheses on $A$ and those needed in the deterministic setting, where
it is enough to prove that $A+\beta$ is maximal monotone in $H$ to get
well-posedness of the nonlinear equation, for any right-hand side
belonging to $L^1(0,T;H)$. Probably the most widely used criterion for
the maximal monotonicity of the sum of two maximal monotone operators
on $H$, at least with applications to PDE in mind, is the following:
let $F$ be a maximal monotone operator on $H$ and $\varphi$ a lower
semi-continuous proper convex function on $H$. If
\begin{equation}  \label{eq:smm}
\varphi\bigl((I+\lambda F)^{-1}u\bigr) \leq \varphi(u) + C\lambda
\qquad \forall \lambda>0, \; \forall u \in \dom(\varphi),
\end{equation}
then $F+\partial\varphi$ is maximal monotone (see \cite[Theorem~9,
p.~108]{Bre-mm}). In the case of semilinear perturbations of the
Laplacian of the type $-\Delta + \beta$, this result is used as
follows: let $\varphi$ be such that $-\Delta=\partial\varphi$, and
\[
\psi: u \mapsto
\begin{cases}
\displaystyle \int_D j(u)\,dx,  & \text{ if } j(u) \in L^1(D),\\[8pt]
+\infty, & \text{ if } j(u) \not\in L^1(D).  
\end{cases}
\]
Then $\psi:H \to \erre \cup \{+\infty\}$ is proper convex lower
semicontinuous, and $F:=\partial\psi$ is maximal monotone, with
$F(u)=\beta(u)$ a.e. for all $u \in H$ such that $j(u) \in L^1(D)$.
Then one has, recalling that $(I+\lambda\beta)^{-1}$ is a contraction
on $\erre$,
\begin{align*}
\varphi\bigl((I+\lambda F)^{-1}u\bigr) 
&= \int_D \abs[\big]{\nabla (I+\lambda\beta)^{-1}u}^2\,dx\\
&\leq \int_D \abs{\nabla u}^2\,dx = \varphi(u),
\end{align*}
so that \eqref{eq:smm} is satisfied, and $-\Delta + \beta$ is maximal
monotone. If one replaces $-\Delta$ with a general positive
self-adjoint operator $A$ on $H$, it is not clear how to adapt such
reasoning. However, if we assume that $A$ is the generator of a
symmetric Dirichlet form $\cE$ on $H$, then \eqref{eq:smm} is
satisfied, with $C=0$ and $\varphi=\cE$. This follows from the fact
that $(I+\lambda \beta)^{-1}$ is a normal contraction on $\erre$ and
that, for any normal contraction $T$ on $\erre$, $u \in \dom(\cE)$
implies $Tu \in \dom(\cE)$ and $\cE(Tu,Tu) \leq \cE(u,u)$, a proof of
which can be found, e.g., in \cite[Theorem~4.12, p.~36]{MR:DF}.

On the other hand, if $A$ is maximal monotone but not self-adjoint, we
cannot express it as the subdifferential of a convex function on
$H$. Hence we are led to ``dualize'' the previous argument, i.e. we
can try to show that
\[
\psi\bigl((I+\lambda A)^{-1}u\bigr) \leq \psi(u) + C\lambda
\qquad \forall \lambda>0, \; \forall u \in \dom(\varphi).
\]
Knowing only that the resolvent is a contraction does not seem enough
to proceed. However, if we assume that the resolvent is sub-Markovian,
we can apply Jensen's inequality (see Lemma~\ref{lm:J} below), so that
\[
j\bigl((I+\lambda A)^{-1}u\bigr) \leq (I+\lambda A)^{-1} j(u),
\]
hence, integrating,
\[
\psi\bigl((I+\lambda A)^{-1}u\bigr) =
\int_D j\bigl((I+\lambda A)^{-1}u\bigr)\,dx 
\leq \int_D (I+\lambda A)^{-1} j(u)\,dx.
\]
Assuming also that the resolvent is contracting in $L^1$, we obtain
$\psi\bigl((I+\lambda A)^{-1}u\bigr) \leq \psi(u)$, hence that
$A+\beta$ is maximal monotone in $H$. Recall that $A$ is contracting
in $L^1$ if it is the generator of a (nonsymmetric) Dirichlet form.
It results from this discussion that our conditions (ii) and (iii) on
$A$ are not restrictive and are probably close to optimal, while the
ultracontractivity condition (iv) is completely superfluous in the
deterministic setting. Moreover, while condition (i') is always
satisfied if $A$ is self-adjoint, it is equally superfluous in the
deterministic case if $A$ is non-symmetric.

\medskip

Let us now comment on the Lipschitz continuity assumption on $B$. It
is natural to ask whether a well-posedness result analogous to
Theorem~\ref{th:main} holds under the weaker assumption that $B$ is
progressively measurable, linearly growing, and just locally Lipschitz
continuous, i.e. assuming that there exists a sequence $(L_B^n)_n$ of
positive real numbers such that
\[
\norm[\big]{B(\omega,t,x) - B(\omega,t,y)}_{\cL^2(U,H)} 
\leq L_B^n \norm{x-y}_H
\]
for every $(\omega,t) \in \Omega \times [0,T]$ and $x,y\in H$ with
$\norm{x}_H, \norm{y}_H \leq n$, for every $n\in\enne$. In this case,
introducing the globally Lipschitz continuous truncated operators
\[
  B_n:\Omega\times[0,T]\times H\to\cL^2(U,H), \qquad B_n(\omega,t,x)
  := B(\omega,t, nPx),
\]
for all $n \in \enne$, where $P:H\to H$ is the projection on the
closed unit ball in $H$, the stochastic evolution equation
\[
  dX_n + AX_n\,dt + \beta(X_n)\,dt \ni B_n(t,X_n)\,dW, \qquad
  X_n(0)=X_0,
\]
is well-posed in $\mathscr{J}$ for all $n \in \enne$. One would now
expect to be able to construct a global solution by suitably
``gluing'' the solutions $(X_n,\xi_n)$. In fact, this technique has
been successfully applied in several situations (cf.,
e.g.,~\cite{beam,KvN2,vNVW}): the key argument is to introduce the
sequence of stopping times $(\tau_n)_n$ defined as
\[
  \tau_n := \inf \bigl\{ t\in[0,T]: \norm{X_n(t)} \geq n \bigr\}
  \wedge T,
\]
and to show that, for any $m > n$, one has $X_m = X_n$ on
\[
  [\![0,\tau_n]\!]:=\bigl\{ (\omega,t) \in \Omega \times [0,T] :\, 0
  \leq t \leq \tau_n(\omega) \bigr\}.
\]
For this construction to work, it seems essential to assume that $X_n$
has continuous trajectories for all $n \in \enne$ (as is the case in
\textsl{op.~cit.}). However, in our case, we only know that the
trajectories of $X_n$ are weakly continuous in $H$, hence the above
construction does not seem to work. On the other hand, we conjecture
that strong solutions in $\mathscr{J}$ to \eqref{eq:0} are indeed
pathwise continuous under suitable polynomial boundedness assumption
on $\beta$, and that, in this case, equations with locally Lipschitz
diffusion coefficient can be shown to be well-posed. This will be
treated in forthcoming work.  We conclude remarking that such a
well-posedness result for semilinear equations with polynomially
growing drift does not follow from the classical variational approach
(see, e.g., \cite[Example~5.1.8]{LiuRo}).

\ifbozza\newpage\else\fi
\section{Preliminaries}
\label{sec:prelim}
We collect, for the reader's convenience, several notions and results
that we are going to use in the following sections.

\subsection{Convex analysis and monotone operators}
\label{ssec:cvx}
We recall basic concepts of convex analysis and their connections with
the theory of maximal monotone operators. We limit ourselves to the
case of functions (and operators) defined on the real line, as we will
not need the general setting of Banach spaces. For a comprehensive
treatment we refer, e.g., to \cite{barbu,Bmax,lema}.

A graph $\gamma$ in $\erre \times \erre$ is called monotone if
\[
(x_1-x_2)(y_1-y_2) \geq 0
\]
for all $(x_1,y_1)$, $(x_2,y_2) \in \gamma$. If $\gamma$ is maximal in
the family of monotone subsets of $\erre \times \erre$, endowed with
the partial order relation of set inclusion, then it is said to be
maximal monotone. In other words, $\gamma$ is maximal monotone if it
does not admit any proper monotone extension. This maximality property
is equivalent to the range condition
\[
\mathsf{R}(I + \lambda \gamma) = \erre \qquad \forall \lambda>0,
\]
where $I$ stands for the identity function. Monotonicity implies that
the inverse $(I+\lambda\gamma)^{-1}$, called the resolvent of
$\gamma$, is single-valued (hence a function, not just a graph) and
contracting. Moreover, $(I+\lambda\gamma)^{-1}$ converges pointwise to
the projection on the closed convex set $\overline{\dom(\gamma)}$ as
$\lambda \to 0$. An essential tool is the Yosida regularization
$\gamma_\lambda:\erre \to \erre$, defined as
\[
\gamma_\lambda := \frac{1}{\lambda}\bigl( I - (I+\lambda\gamma)^{-1} \bigr),
\qquad \lambda>0.
\]
The following properties will be used extensively:
\begin{itemize}
\item[(a)] $\gamma_\lambda$ is monotone and Lipschitz continuous, with
  Lipschitz constant bounded by $1/\lambda$;
\item[(b)] $\gamma_\lambda \in \gamma \circ (I+\lambda\gamma)^{-1}$.
\end{itemize}

\medskip

Let $\varphi: \erre \to \erre \cup \{+\infty\}$ be a function not
identically equal to $+\infty$ (i.e., proper), convex and
lower-semicontinuous.
Denoting the set of subsets of $\erre$ by $\mathfrak{P}(\erre)$, the map
\begin{align*}
  \partial\varphi: \erre &\longrightarrow \mathfrak{P}(\erre)\\
  x &\longmapsto \bigl\{z \in \erre: \, \varphi(y) - \varphi(x) \geq z(y-x)
\quad \forall y \in \erre \bigr\}
\end{align*}
is called the subdifferential of $\varphi$. The multivalued map
$\gamma:=\partial\varphi$, that can equivalently be considered as a graph in
$\erre \times \erre$, is maximal monotone. Conversely, every maximal
monotone graph of $\erre \times \erre$ is the subdifferential of a
convex proper function, which is, roughly speaking, its indefinite
integral.

The Moreau-Yosida regularization of $\varphi$ is the convex
differentiable function $\varphi_\lambda:\erre \to \erre$ defined by
\[
  \varphi_\lambda(x) := \inf_{y \in \erre} \Bigl( \varphi(y) +
  \frac{\abs{x-y}^2}{2\lambda} \Bigr), \qquad \lambda>0.
\]
It enjoys the following fundamental properties:
\begin{itemize}
\item[(c)] $\varphi'_\lambda = \gamma_\lambda$, where $\gamma_\lambda$
denotes the Yosida regularization of $\gamma=\partial\varphi$;
\item[(d)] $\varphi_\lambda$ converges pointwise to $\varphi$ from
  below as $\lambda \to 0$;
\end{itemize}
The (Fenchel-Legendre) conjugate of $\varphi$ is
the proper convex lower-semicontinuous function
$\varphi^*: \erre \to \erre \cup \{+\infty\}$ defined as
\[
\varphi^*: x \mapsto \sup_{y\in\erre} \bigl( xy - \varphi(y) \bigr).
\]
The Young inequality
\[
xy \leq \varphi(y) + \varphi^*(x) \qquad \forall x,y \in \erre
\]
follows immediately from the definition. The following properties will be particularly useful:
\begin{itemize}
\item[(e)] equality holds in the Young inequality if and only if
  $x \in \partial \varphi(y)$;
\item[(f)] if $\dom(\gamma)=\erre$, then $\varphi^*$ is superlinear
  at infinity, i.e.
  \[
  \lim_{\abs{r} \to \infty} \frac{\varphi^*(r)}{\abs{r}} = +\infty.
  \]
\end{itemize}

We shall also need a result about passing to the limit ``within''
maximal monotone graphs due to Br\'ezis, see \cite[Theorem~18,
p.~126]{Bre-mm}.
\begin{lemma}  \label{lm:Brezis}
  Let $\gamma$ be a maximal monotone graph in $\erre \times \erre$ with
  $\dom(\gamma)=\erre$ and $0 \in \gamma(0)$. Assume that the sequences
  $(y_n)_{n\in\mathbb{N}}$, $(g_n)_{n\in\mathbb{N}}$ of real-valued
  measurable functions on a finite measure space $(Y,\mathscr{A},\mu)$
  are such that $y_n \to y$ $\mu$-a.e. as $n\to\infty$,
  $g_n \in \gamma(y_n)$ $\mu$-a.e. for all $n\in\mathbb{N}$, and
  $(g_ny_n)$ is a bounded subset of $L^1(Y,\mathscr{A},\mu)$. Then
  there exists $g \in L^1(Y,\mathscr{A},\mu)$ and a subsequence $n'$
  such that $g_{n'} \to g$ weakly in $L^1(Y,\mathscr{A},\mu)$ as
  $n' \to \infty$ and $g \in \gamma(y)$ $\mu$-almost everywhere.
\end{lemma}

Finally, we recall a simplified version of an ``abstract'' Jensen's
inequality, due to Haase (see \cite[Theorem~3.4]{Haa07}), that will be
used to prove a priori estimates for convex functionals of stochastic
processes.
\begin{lemma}
  \label{lm:J}
  Let $(Y,\mathscr{A},\mu)$, $(Z,\mathscr{B},\nu)$ be measure spaces,
  $E \subset L^0(Y,\mathscr{A},\mu)$ a Banach function space, and
  \[
  T: E \longrightarrow L^0(Z,\mathscr{B},\nu)
  \]
  a linear continuous sub-Markovian operator. Moreover, let
  $\varphi:\erre \to \mathopen[0,\infty\mathclose[$ be a convex lower
  semicontinuous function with $\varphi(0)=0$. Then
  \[
  \varphi(Tf) \leq T\varphi(f)
  \]
  for all $f \in  E$ such that $\varphi(f) \in E$.
\end{lemma}

\subsection{Hilbert-Schmidt operators}
Let us recall now some standard facts about linear maps. We recall
that the space of continuous linear operators from a Banach space $E$
to another one $F$, equipped with the strong operator topology, is
denoted by $\cL_s(E,F)$. If $E$ and $F$ are Hilbert spaces, the space
of Hilbert-Schmidt operators $\cL^2(E,F)$ is an operator ideal,
\luca{in particular} it is stable with respect to pre-composition as
well as post-composition with continuous linear operators: if $E'$ and
$F'$ are also Hilbert spaces, and
\[
E' \xrightarrow{\;R\;} E \xrightarrow{\;T\;} F \xrightarrow{\;L\;} F',
\]
with $R$ and $L$ continuous linear operators, then
$LTR \in \cL^2(E',F')$,\footnote{One may say, in a shorter but perhaps
  cryptic way, that $\cL^2$ is functorial, more precisely that
  $\cL^2(E,\cdot)$ and $\cL^2(\cdot,F)$ are a covariant and a
  contravariant functor, respectively.} with
\[
\norm[\big]{LTR}_{\cL^2(E',F')} \leq 
\norm[\big]{L}_{\cL(F,F')} \norm[\big]{T}_{\cL^2(E,F)} \norm[\big]{R}_{\cL(E',E)} 
\]
(see, e.g., \cite[p.~V.52]{Bbk:EVT}). It follows from these properties
that, for any $T \in \cL^2(E,F)$, the mapping
\begin{align*}
\Phi_T: \cL_s(F,F') &\longrightarrow \cL^2(E,F')\\
L \longmapsto LT
\end{align*}
is continuous: $L_n \to L$ in $\cL_s(F,F')$ implies that $L_nT \to LT$
in $\cL^2(E,F')$. If $E$ and $F$ are separable, then $\cL^2(E,F)$ is
itself a separable Hilbert space.

\begin{lemma}  \label{lm:DY}
  If $G$ is a progressively measurable $\cL^2(U,H)$-valued process such that
  \[
    \E\int_0^T\norm[\big]{G(s)}^2_{\cL^2(U,H)}\,ds<\infty
  \]
  and $F$ is a progressively measurable $H$-valued process such that
  $\E (F^*_T)^2<\infty$, 
  then, for any $\varepsilon>0$,
  \[
  \E \bigl((FG)\cdot W\bigr)_T^* \leq \varepsilon
  \E \bigl(F_T^*\bigr)^2 
  + N(\varepsilon) \E\int_0^T \norm[\big]{G(s)}^2_{\cL^2(U,H)}\,ds.
  \]
\end{lemma}
\begin{proof}
  By the ideal property of Hilbert-Schmidt operators, one has
  \begin{align*}
  \norm[\big]{F(s)G(s)}_{\cL^2(U,\erre)} &\leq
  \norm[\big]{F(s)}_{H} \norm[\big]{G(s)}_{\cL^2(U,H)}\\
  &\leq (F^*_T) \norm[\big]{G(s)}_{\cL^2(U,H)}
  \end{align*}
  for all $s \in [0,T]$, hence
  \[
  \int_0^T \norm[\big]{F(s)G(s)}^2_{\cL^2(U,\erre)}\,ds \leq
  (F^*_T)^2
  \int_0^T \norm[\big]{G(s)}^2_{\cL^2(U,H)}\,ds,
  \]
  where the right-hand side is finite $\P$-a.s. thanks to the
  assumptions on $F$ and $G$. Then $(FG) \cdot W$ is a
  local martingale, for which Davis' inequality yields
  \begin{align*}
  \E \bigl((FG)\cdot W\bigr)_T^* &\lesssim 
  \E\bigl[(FG)\cdot W,(FG)\cdot W\bigr]_T^{1/2}\\
  &= \E\biggl( \int_0^T 
  \norm[\big]{F(s)G(s)}^2_{\cL^2(U,\erre)}\,ds \biggr)^{1/2}\\
  &\leq \E (F^*_T)
  \biggl( \int_0^T \norm[\big]{G(s)}^2_{\cL^2(U,H)}\,ds \biggr)^{1/2}.
  \end{align*}
  The proof is finished invoking the elementary inequality
  \[
  ab \leq \frac12 \bigl( \varepsilon a^2 + \frac{1}{\varepsilon} b^2\bigr)
  \qquad \forall a, b \in \erre.
  \qedhere
  \]
\end{proof}

\subsection{Continuity and compactness in spaces of vector-valued
  functions}
The following result by Strauss, see \cite[Theorem~2.1]{Strauss},
provides sufficient conditions for a vector-valued function to be
weakly continuous. It will be used to establish the pathwise weak
continuity of solutions to several stochastic equations. We recall
that, given a Banach space $E$ and an interval $I \subseteq \erre$, the
space of weakly continuous functions from $I$ to $E$ is denoted
by $C_w(I;E)$.
\begin{lemma}
  \label{lm:Strauss}
  Let $E$ and $F$ be Banach spaces such that $E$ is dense in $F$,
  $E \embed F$, and $E$ is reflexive. Then
  \[
  L^\infty(0,T;E) \cap C_w([0,T];F) = C_w([0,T];E).
  \]
\end{lemma}

The next result is a classical integration-by-parts formula, whose
proof can be found, for instance, in \cite[{\S}1.3]{barbu}. Let
$\mathcal V$ and $\mathcal H$ be Hilbert spaces such that
$\mathcal V \embed\mathcal H \embed\mathcal V^*$, and denote by
$W(a,b;\mathcal V)$ the set of functions $u \in L^2(a,b;\mathcal V)$
such that $u' \in L^2(a,b;\mathcal V^*)$, where the derivative $u'$ is
meant in the sense of $\mathcal V^*$-valued distributions. The duality
of $\mathcal V$ and $\mathcal V^*$ as well as the scalar product of
$\mathcal H$ will be denoted by $\ip{\cdot}{\cdot}$.
\begin{lemma}   \label{lm:AC} 
  Let $u \in W(a,b;\mathcal V)$. Then there exists
  $\tilde{u} \in C([a,b];\mathcal H)$ such that $u(t)=\tilde{u}(t)$
  for almost all $t \in [a,b]$. Moreover, for any
  $v \in W(a,b;\mathcal V)$, $\ip{u}{v}$ is absolutely continuous on
  $[a,b]$ and
  \[
    \frac{d}{dt} \ip[\big]{u(t)}{v(t)} = \ip[\big]{u'(t)}{v(t)} +
    \ip[\big]{u(t)}{v'(t)}.
  \]
\end{lemma}

The following compactness criterion is due to Simon, see \cite[Corollary~4,
p.~85]{Simon}.
\begin{lemma}
  \label{lm:Simon}
  Let $E_1$, $E_2$, $E_3$ be Banach spaces such that $E_1 \embed E_2$
  and $E_2 \embed E_3$ compactly. \luca{Assume that
  $F$ is a bounded subset of $L^p(0,T;E_1)\cap W^{1,1}(0,T; E_3)$
  for some $p \geq 1$.} Then
  $F$ is relatively compact in $L^p(0,T;E_2)$.
\end{lemma}

\ifbozza\newpage\else\fi
\section{Well-posedness for a regularized equation}
\label{sec:reg}
Let $V_0$ be a separable Hilbert space such that $V_0$ is a dense
subset of $V$, $V_0 \embed V$, and $V_0 \embed L^\infty(D)$. The goal
of this section is to establish existence and uniqueness of solutions
to the stochastic evolution equation
\begin{equation}
\label{eq:V0}
dX(t) + AX(t)\,dt + \beta(X(t))\,dt \ni B(t)\,dW(t),
\qquad X(0)=X_0,
\end{equation}
where $B$ is an $\cL^2(U,V_0)$-valued process. In particular, this
stochastic equation can be interpreted as a version of \eqref{eq:0}
with additive and more regular noise.

\begin{prop}\label{prop:V0}
  Assume that $X_0\in L^2(\Omega,\cF_0,\P; H)$ and that
  \[
  B \in L^2(\Omega;L^2(0,T;\cL^2(U,V_0)))
  \]
  is measurable and adapted.  Then equation \eqref{eq:V0} admits a
  unique strong solution $(X,\xi)$ such that
  \begin{gather*}
    X \in L^2(\Omega;L^\infty(0,T;H)) \cap L^2(\Omega;L^2(0,T;V)),\\
    j(X)+j^*(\xi)\in L^1((0,T)\times D) \qquad \text{$\P$-almost surely}.
  \end{gather*}
  Moreover, $X(\omega,\cdot) \in C_w([0,T];H)$ for $\P$-almost all
  $\omega \in \Omega$.
\end{prop}

The rest of this section is devoted to the proof of Proposition
\ref{prop:V0}, which is structured as a follows: we consider a
regularized version of \eqref{eq:V0}, where the nonlinear term $\beta$
is replaced by its Yosida approximation, and obtain suitable a priori
estimates, both pathwise and in expectation. Taking limits in
appropriate topologies of the solutions to these regularized
equations, we construct solutions to \eqref{eq:V0}, that are finally
shown to be unique.

\medskip
Let
\[
\beta_\lambda := \frac{1}{\lambda}\bigl( I - (I+\lambda\beta)^{-1} \bigr),
\qquad \lambda>0,
\]
be the Yosida approximation of $\beta$, and consider the regularized
equation
\[
dX_\lambda(t) + AX_\lambda(t)\,dt + \beta_\lambda(X_\lambda(t))\,dt =
B(t)\,dW(t), \qquad X_\lambda(0)=X_0.
\]
Since $\beta_\lambda$ is monotone and Lipschitz continuous, it is easy
to check that the operator $A+\beta_\lambda$ satisfies, for any
$\lambda>0$, the classical conditions of Pardoux, Krylov and
Rozovski\u{\i} \cite{KR-spde,Pard}. For completeness, a proof is given
next.  \luca{
\begin{lemma}
  Let $\lambda>0$. The operator
  $A_\lambda:=A+\beta_\lambda:V \to V^*$ satisfies the
  following conditions:
  \begin{itemize}
  \item[\emph{(i)}] $A_\lambda$ is hemicontinuous, i.e. the map $\erre
    \ni \eta \mapsto \ip{A_\lambda(u+\eta v)}{x}$ is continuous for
    all $u$, $v$, $x \in V$;
  \item[\emph{(ii)}] $A_\lambda$ is monotone, i.e.  $\ip{A_\lambda u -
      A_\lambda v}{u-v} \geq 0$ for all $u$, $v \in V$;
  \item[\emph{(iii)}] $A_\lambda$ is coercive, i.e. there exists a
    constant $C_1>0$ such that $\ip{A_\lambda v}{v} \geq C_1
    \norm{v}_{V}^2$ for all $v \in V$;
  \item[\emph{(iv)}] $A_\lambda$ is bounded, i.e. there exists a
    constant $C_2>0$ such that $\norm{A_\lambda v}_{V^*} \leq
    C_2\norm{v}_{V}$ for all $v \in V$.
  \end{itemize}
\end{lemma}
\begin{proof}
  (i) For any $u,v,x\in V$, one has
  \[
  \ip{A_\lambda(u+\eta v)}{x} =
  \ip{Au}{x} + \eta \ip{Av}{x} + \int_D \beta_\lambda(u+\eta v)x.
  \]
  It clearly suffices to check that the last term depends continuously
  on $\eta$, which follows immediately by the Lipschitz continuity of
  $\beta_\lambda$. (ii) Since both $A$ and $\beta_\lambda$ are
  monotone, one has
  \[
  \ip{A_\lambda u-A_\lambda v}{u-v} =
  \ip{Au - Av}{u-v} 
  + \int_D (\beta_\lambda(u)-\beta_\lambda(v) (u-v) \geq 0.
  \]
  (iii) Similarly, since $0 \in \beta(0)$ implies
  $\beta_\lambda(0)=0$, coercivity of $A$ and monotonicity of
  $\beta_\lambda$ imply
  \[
  \ip{A_\lambda v}{v} = \ip{Av}{v} 
  + \int_D \beta_\lambda(v)v \geq \ip{Av}{v} \geq C\norm{v}_V^2
  \]
  (in particular, $C_1$ can be chosen equal to $C$, the coercivity
  constant of $A$ itself).  (iv) Using again the fact that
  $\beta_\lambda(0)=0$, and recalling that $\beta_\lambda$ is
  Lipschitz continuous with Lipschitz constant bounded by $1/\lambda$,
  one has
  \begin{align*}
  \ip{A_\lambda v}{u}&=
  \ip{Av}{u} + \int_D \beta_\lambda(v)u 
  \leq \norm{Av}_{V^*} \norm{u}_V +\frac1\lambda \norm{v}_{H}\norm{u}_{H}\\
  &\leq \bigl( \norm{A}_{\cL(V,V^*)} + k/\lambda \bigr) \norm{v}_{V}\norm{u}_{V},
  \end{align*}
  where $k$ is the norm of the continuous embedding $\iota: V \to H$.
\end{proof}
}
Hence \eqref{eq:reg} admits a
unique variational solution, that is, there exists a unique adapted
process
\[
X_\lambda \in L^2(\Omega;C([0,T];H)) \cap L^2(\Omega;L^2(0,T;V))
\]
such that, in $V^*$,
\begin{equation}\label{eq:reg}
  X_\lambda(t) + \int_0^t AX_\lambda(s)\,ds + \int_0^t \beta_\lambda(X_\lambda(s))\,ds
  = X_0 + \int_0^t B(s)\,dW(s)  
\end{equation}
for all $t \in [0,T]$.

In the next lemmata we establish a priori estimates for $X_\lambda$
and $\beta_\lambda(X_\lambda)$. We begin with a pathwise estimate.
\begin{lemma}\label{lm:stp}
  There exists $\Omega' \subseteq \Omega$ with $\P(\Omega')=1$ and
  $M:\Omega' \to \erre$ such that
  \[
  \norm[\big]{X_\lambda(\omega)}^2_{C([0,T];H) \cap L^2(0,T;V)} + 
  \norm[\big]{j_\lambda(X_\lambda(\omega))}_{L^1(0,T;L^1(D))} < M(\omega)
  \]
  for all $\omega \in \Omega'$.
\end{lemma}
\begin{proof}
  Setting $Y_\lambda := X_\lambda - B \cdot W$, It\^o's
  formula\footnote{Whenever we refer to It\^o's formula, we shall
    always mean the version in \cite{KR-spde}.} yields
  \[
  \norm[\big]{Y_\lambda(t)}^2_H 
  + 2\int_0^t \ip[\big]{AX_\lambda(s)}{Y_\lambda(s)}\,ds
  + 2\int_0^t \ip[\big]{\beta_\lambda(X_\lambda)}{Y_\lambda(s)}\,ds
  = \norm[\big]{X_0}^2_H,
  \]
  where
  $\norm{X_\lambda}_H \leq \norm{Y_\lambda}_H + \norm{B\cdot W}_H$ by
  the triangle inequality, hence
  \[
  \norm{Y_\lambda(t)}^2_H \geq \frac12 \norm{X_\lambda(t)}_H^2 -
  \norm{B \cdot W(t)}_H^2.
  \]
  Moreover, writing
  $ \ip{AX_\lambda}{Y_\lambda} = \ip{AX_\lambda}{X_\lambda} -
  \ip{AX_\lambda}{B \cdot W}$, one has
  \[
  \ip{AX_\lambda}{X_\lambda} \geq C \norm{X_\lambda}_V^2
  \]
  by the coercivity of $A$, and
  \begin{align*}
  \ip{AX_\lambda}{B \cdot W} &\leq \norm{A}_{\cL(V,V^*)} \norm{X_\lambda}_V
  \norm{B \cdot W}_V\\
  &\leq \frac12 C \norm{X_\lambda}_V^2 + \frac{1}{2\varepsilon}
  \norm{B \cdot W}_V^2,
  \end{align*}
  where we have used the elementary inequality
  $ab \leq \frac12(\varepsilon a^2+b^2/\varepsilon)$ for all $a,b \in \erre$,
  with $\varepsilon:=C\norm{A}_{\cL(V,V^*)}^{-2}$. Then
  \[
  \ip{AX_\lambda}{Y_\lambda} \geq 
  \frac12 C \norm{X_\lambda}_V^2 
  - \frac{1}{2\varepsilon} \norm{B \cdot W}_V^2,
  \]
  so that
  \[
  2\int_0^t \ip[\big]{AX_\lambda(s)}{Y_\lambda(s)}\,ds \geq
  C \int_0^t \norm{X_\lambda(s)}_V^2\,ds 
  - \frac{1}{\varepsilon} \int_0^t \norm{B\cdot W(s)}_V^2\,ds
  \]
  and
  \begin{equation}
  \label{eq:pippo}
  \begin{split}
  &\frac12 \norm{X_\lambda(t)}_H^2 + C \int_0^t  \norm{X_\lambda(s)}_V^2\,ds
  + 2 \int_0^t \ip[\big]{\beta_\lambda(X_\lambda(s))}{Y_\lambda(s)}\,ds\\
  &\hspace{3em} \leq \norm{X_0}_H^2 + \norm{B\cdot W(t)}_H^2
  + \frac1\varepsilon \int_0^t \norm{B\cdot W(s)}_V^2\,ds.
  \end{split}
  \end{equation}
  Let $j_\lambda$ be the Moreau-Yosida regularization of $j$, that is
  \[
  j_\lambda(x) := \inf_{y \in \erre} \Bigl( 
    j(y) + \frac{\abs{x-y}^2}{2\lambda} \Bigr), \qquad \lambda>0.
  \]
  We recall that $j_\lambda$ is a convex, proper differentiable
  function, with $j'_\lambda = \beta_\lambda$, that converges
  pointwise to $j$ from below. In particular,
  \[
  \beta_\lambda(x)(x-y) \geq j_\lambda(x)-j_\lambda(y)
  \geq j_\lambda(x)-j(y) \qquad \forall x,y \in \erre.
  \]
  This implies
  \begin{align*}
  \int_0^t \ip[\big]{\beta_\lambda(X_\lambda(s))}{Y_\lambda(s)}\,ds
  &= \int_0^t\!\!\int_D 
     \beta_\lambda(X_\lambda(s,x))(X_\lambda(s,x) - B \cdot W(s,x))\,dx\,ds\\
  &\geq \int_0^t\!\!\int_D j_\lambda(X_\lambda(s,x))\,dx\,ds
  - \int_0^t\!\!\int_D j(B\cdot W(s,x))\,dx\,ds,
  \end{align*}
  hence also
  \begin{align*}
  &\frac12 \norm{X_\lambda(t)}_H^2 + C \int_0^t  \norm{X_\lambda(s)}_V^2\,ds
  + 2 \int_0^t\!\!\int_D j_\lambda(X_\lambda(s,x))\,dx\,ds\\
  &\hspace{3em} \leq \norm{X_0}_H^2 + \norm{B\cdot W(t)}_H^2
  + \frac1\varepsilon \int_0^t \norm{B\cdot W(s)}_V^2\,ds\\
  &\hspace{3em}\quad + 2\int_0^t\!\!\int_D j(B\cdot W(s,x))\,dx\,ds.
  \end{align*}
  Taking the supremum with respect to $t$ yields
  \luca{
  \begin{align*}
  &\norm[\big]{X_\lambda}_{C([0,T];H)}^2
  + \norm[\big]{X_\lambda}_{L^2(0,T;V)}^2
  + \norm[\big]{j_\lambda(X_\lambda)}_{L^1(0,T;L^1(D))}\\
  &\hspace{3em} \lesssim \norm[\big]{X_0}_H^2
  + \norm[\big]{B\cdot W}_{C([0,T];H)}^2
  + \norm[\big]{B\cdot W}_{L^2(0,T;V)}^2
  + \norm[\big]{j(B\cdot W)}_{L^1(0,T;L^1(D))},
  \end{align*}}
  where the implicit constant depends only on the operator norm of
  $A$.  It follows by It\^o's isometry and Doob's inequality that
  \[
  \norm[\big]{B \cdot W}_{L^2(\Omega;C([0,T];V_0))} \lesssim
  \norm[\big]{B}_{\luca{L^2(\Omega;L^2(0,T;\cL^2(U,V_0)))}},
  \]
  where the right-hand side is finite by assumption, hence, recalling
  that $V_0$ is continuously embedded in $V$,
  \[
  \norm[\big]{B\cdot W}_{C([0,T];H)}
  + \norm[\big]{B\cdot W}_{L^2(0,T;V)} \lesssim_T
  \norm[\big]{B\cdot W}_{C([0,T];V_0)}.
  \]
  Analogously, denoting the norm of the continuous embedding
  $\iota: V_0 \to L^\infty(D)$ by $k$, one has, recalling that $j$ is
  symmetric and increasing on $\erre_+$,
  \[
  \norm[\big]{j(B\cdot W(t)}_{L^1(D)} \lesssim_{\abs{D}}
  j\bigl(\norm{B\cdot W(t)}_{L^\infty(D)}\bigr) \leq
  j\bigl(k\norm{B\cdot W(t)}_{V_0}\bigr),
  \]
  for all $t \in [0,T]$, hence
  \[
  \norm[\big]{j(B\cdot W)}_{L^1(0,T;L^1(D))}
  \lesssim_{\abs{D},T} j\bigl(k\norm{B\cdot W}_{C([0,T];V_0)}\bigr).
  \]
  The proof is complete choosing $\Omega' \subset \Omega$ such that
  $\norm{X_0(\omega)}_H$ and $\norm{B\cdot
    W(\omega)}_{C([0,T];V_0)}$ are finite for all $\omega \in
  \Omega'$, \luca{and defining $M:\Omega'\rightarrow\erre$ as
  \[
  M:=
  \norm[\big]{X_0}_H^2
  + \norm[\big]{B\cdot W}_{C([0,T];H)}^2
  + \norm[\big]{B\cdot W}_{L^2(0,T;V)}^2
  + \norm[\big]{j(B\cdot W)}_{L^1(0,T;L^1(D))}\,.
  \qedhere
  \]}
\end{proof}

\begin{rmk}\label{rmk:rtr}
  The above estimates can be obtained by purely deterministic
  arguments, without invoking It\^o's formula. In fact, note that
  equation \eqref{eq:reg} can equivalently be written as
  \[
  Y_\lambda(t) 
  + \int_0^t \bigl( AX_\lambda(s) + \beta_\lambda(X_\lambda(s))\bigr)\,ds = 0.
  \]
  One has $Y_\lambda \in L^2(0,T;V)$, which follows at
  once by the properties of $X_\lambda$ and by
  $B \cdot W \in L^2(\Omega;C([0,T];V_0))$. Similarly, since
  $AX_\lambda$ and $\beta_\lambda(X_\lambda)$ belong to
  $L^2(\Omega;L^2(0,T;V^*))$, one also has, by the previous identity,
  $Y'_\lambda \in L^2(0,T;V^*)$. In particular, there
  exists $\Omega' \subset \Omega$, with $\P(\Omega')=1$, such that
  \[
  Y_\lambda(\omega) \in L^2(0,T;V), \qquad Y'_\lambda(\omega) \in
  L^2(0,T;V^*) \qquad \forall \omega \in \Omega'.
  \]
  Lemma \ref{lm:AC} then yields
  \[
  \frac{1}{2} \norm[\big]{Y_\lambda(t)}^2_H 
  + \int_0^t \ip[\big]{AX_\lambda(s)}{Y_\lambda(s)}\,ds
  + \int_0^t \ip[\big]{\beta_\lambda(X_\lambda)}{Y_\lambda(s)}\,ds
  = \frac{1}{2} \norm[\big]{X_0}^2_H.
  \]
\end{rmk}

\begin{lemma}  \label{lm:aspetta}
  There exists a constant $N>0$ such that
  \begin{align*}
  &\norm[\big]{X_\lambda}_{L^2(\Omega;C([0,T];H))}^2 + 
  \norm[\big]{X_\lambda}_{L^2(\Omega;L^2(0,T;V)}^2 +
  \norm[\big]{\beta_\lambda(X_\lambda)X_\lambda}_{L^1(\Omega;L^1(0,T;L^1(D)))}\\
  &\hspace{3em} < N\Bigl(
  \norm[\big]{X_0}^2_{L^2(\Omega;H)} 
  + \norm[\big]{B}^2_{L^2(\Omega;L^2(0,T;\cL^2(U,H)))} \Bigr).
  \end{align*}
\end{lemma}
\begin{proof}
  It\^o's formula yields
  \begin{align*}
  &\norm[\big]{X_\lambda(t)}_H^2 
  + 2\int_0^t \ip[\big]{AX_\lambda(s)}{X_\lambda(s)}\,ds
  + 2 \int_0^t \ip[\big]{\beta_\lambda(X_\lambda(s))}{X_\lambda(s)}\,ds\\
  &\hspace{3em} = \norm[\big]{X_0}_H^2
  + 2\int_0^t X_\lambda(s) B(s)\,dW(s) 
  + \frac12 \int_0^t \norm[\big]{B(s)}^2_{\cL^2(U,H)}\,ds,
  \end{align*}
  where $X_\lambda$ in the stochastic integral on the right-hand side
  has to be interpreted as taking values in $H^* \simeq H$. The
  coercivity of $A$ and the monotonicity of $\beta_\lambda$ readily
  imply, after taking supremum in time and expectation,
  \begin{align*}
  &\E\norm[\big]{X_\lambda}^2_{C([0,T];H)} 
  + 2C \E\norm[\big]{X_\lambda}^2_{L^2(0,T;V)}
  + \E\int_0^T \ip[\big]{\beta_\lambda(X_\lambda(s))}{X_\lambda(s)}\,ds\\
  &\hspace{5em} \lesssim
  \E\norm[\big]{X_0}_H^2 + \E\norm[\big]{B}^2_{L^2(0,T;\cL^2(U,H))}
  + \E\sup_{t\in[0,T]} \abs[\bigg]{\int_0^t X_\lambda(s) B(s)\,dW(s)},
  \end{align*}
  where, by Lemma \ref{lm:DY},
  \[
  \E\sup_{t\in[0,T]} \abs[\bigg]{\int_0^t X_\lambda(s) B(s)\,dW(s)}
  \leq \varepsilon \E \norm[\big]{X_\lambda}^2_{C([0,T];H)}
  + N(\varepsilon) \E\int_0^T \norm[\big]{B(s)}^2_{\cL^2(U,H)}\,ds
  \]
  for any $\varepsilon>0$, whence the result follows choosing
  $\varepsilon$ small enough.
\end{proof}

We now establish weak compactness properties for the sequence
$(\beta_\lambda(X_\lambda))$.
\begin{lemma}
\label{lm:L1}
  The sequence $(\beta_\lambda(X_\lambda))$ is relatively
  weakly compact in $L^1(\Omega \times (0,T) \times D)$. Moreover,
  there exists a set $\Omega'' \subset \Omega$, with $\P(\Omega'')=1$,
  such that $(\beta_\lambda(X_\lambda(\omega,\cdot))$
  is weakly relatively compact in $L^1((0,T) \times D)$ for all
  $\omega \in \Omega''$.
\end{lemma}
\begin{proof}
  Recalling that, for any $y$, $r \in \erre$, $j(y)+j^*(r)=ry$ if and
  only if $r \in \partial j(y)=\beta(y)$, one has
  \begin{equation}
  \label{eq:pluto}
  j\bigl((I+\lambda\beta)^{-1}x\bigr) + j^*\bigl(\beta_\lambda(x)\bigr)
  = \beta_\lambda(x) (I+\lambda\beta)^{-1}x \leq \beta_\lambda(x) x
  \qquad \forall x \in \erre.
  \end{equation}
  In fact, since $\beta_\lambda \in \beta \circ
  (I+\lambda\beta)^{-1}$, it follows from $\beta=\partial j$ that
  $\beta_\lambda(x) \in \partial
  j\bigl((I+\lambda\beta)^{-1}x\bigr)$. Moreover,
  $\beta\bigl((I+\lambda\beta)^{-1}x\bigr) (I+\lambda\beta)^{-1}x \geq
  0$ by monotonicity of $\beta$, hence the inequality in
  \eqref{eq:pluto} follows since $(I+\lambda\beta)^{-1}$ is a
  contraction. The previous lemma thus implies, thanks to the symmetry
  of $j^*$, that there exists a constant $N$, independent of
  $\lambda$, such that, setting
  \[
  \bar{N}(X_0,B) := N\Bigl(
  \norm[\big]{X_0}^2_{L^2(\Omega;H)} 
  + \norm[\big]{B}^2_{L^2(\Omega;L^2(0,T;\cL^2(U,H)))} \Bigr),
  \]
  one has
  \[
  \E\int_0^T\!\!\int_D j^*\bigl(\abs{\beta_\lambda(X_\lambda)}\bigr)
  \leq \E\int_0^T\!\!\int_D \beta_\lambda(X_\lambda)X_\lambda < \bar{N}(X_0,B).
  \]
  Since $j^*$ is superlinear at infinity, the sequence
  $(\beta_\lambda(X_\lambda))$ is uniformly integrable on
  $\Omega \times (0,T) \times D$ by the de la Vall\'ee-Poussin
  criterion, hence weakly relatively compact in
  $L^1(\Omega \times (0,T) \times D)$ by a well-known theorem of
  Dunford and Pettis. The first assertion is thus proved.

  By \eqref{eq:pippo}, since $Y_\lambda=X_\lambda - B \cdot W$, it
  follows that
  \begin{align*}
  \int_0^t \ip[\big]{\beta_\lambda(X_\lambda(s))}{X_\lambda(s)}\,ds &\lesssim
  \norm{X_0}_H^2 + \norm{B\cdot W(t)}_H^2
  + \int_0^t \norm{B\cdot W(s)}_V^2\,ds\\
  &\quad + \int_0^t \ip[\big]{\beta_\lambda(X_\lambda(s))}{B\cdot W(s)}\,ds,
  \end{align*}
  where, by Young's inequality and convexity (recalling that $j^*(0)=0$),
  \[
  \int_0^t \ip[\big]{\beta_\lambda(X_\lambda(s))}{B\cdot W(s)}\,ds \leq
  \frac12 \int_0^t\!\!\int_D j^*\bigl( \beta_\lambda(X_\lambda)\bigr)
  + \int_0^t\!\!\int_D j(2B \cdot W).
  \]
  Rearranging terms and proceeding as in the (end of the) proof of
  Lemma \ref{lm:stp}, we infer that there exists a set
  $\Omega''\subset\Omega$, with $\P(\Omega'')=1$, and a function
  $M:\Omega'' \to \erre$ such that
  \begin{equation}  \label{eq:plutone} 
    \int_0^T \ip[\big]{\beta_\lambda(X_\lambda(\omega,s))}%
    {X_\lambda(\omega,s)}\,ds < M(\omega) \qquad \forall \omega \in
    \Omega''.
  \end{equation}
  The symmetry of $j^*$ and \eqref{eq:pluto} yield, as before, that,
  for any $\omega \in \Omega''$,
  $(\beta_\lambda(X_\lambda(\omega,\cdot)))$ is weakly
  relatively compact in $L^1((0,T)\times D)$.
\end{proof}

In order to pass to the limit as $\lambda \to 0$, we are going to use
Simon's compactness criterion, i.e. Lemma \ref{lm:Simon}, and Br\'ezis' Lemma
\ref{lm:Brezis}.

\begin{prop}
\label{prop:cvg}
  There exists $\Omega' \subseteq \Omega$, with $\P(\Omega')=1$, such
  that, for any $\omega \in \Omega'$, there exists a subsequence
  $\lambda'=\lambda'(\omega)$ of $\lambda$ such that, as
  $\lambda' \to 0$,
  \begin{align*}
    &X_{\lambda'}(\omega,\cdot) \xrightharpoonup{\;*\;} X(\omega,\cdot) 
    & &\text{in } L^\infty(0,T;H),\\
    &X_{\lambda'}(\omega,\cdot) \xrightharpoonup{\;\phantom{*}\;} 
    X(\omega,\cdot) & &\text{in } L^2(0,T;V),\\
    &X_{\lambda'}(\omega,\cdot) \xrightarrow{\;\phantom{*}\;} 
    X(\omega,\cdot) & &\text{in } L^2(0,T;H),\\
    &\beta_{\lambda'}(X_{\lambda'}(\omega,\cdot)) \xrightharpoonup{\;\phantom{*}\;} 
    \xi(\omega,\cdot) & &\text{in } L^1((0,T) \times D).
  \end{align*}
\end{prop}
\begin{proof}
  The first two convergence statements follow by Lemma \ref{lm:stp},
  and the fourth one follows by Lemma \ref{lm:L1}. Let us show that
  the third convergence statement holds. In the following we omit
  the indication of $\omega$, as no confusion can arise. Setting
  $Y_\lambda=X_\lambda - B \cdot W$, \eqref{eq:reg} can equivalently
  be written as the deterministic equation (with random coefficients)
  on $V^*$
  \[
  Y'_\lambda + AX_\lambda + \beta_\lambda(X_\lambda) = 0,
  \]
  where
  \begin{gather*}
    \norm[\big]{AX_\lambda}_{L^1(0,T;V_0^*)} \lesssim
    \norm[\big]{AX_\lambda}_{L^1(0,T;V^*)} \lesssim
    \norm[\big]{X_\lambda}_{L^1(0,T;V)},\\
    \norm[\big]{\beta_\lambda(X_\lambda)}_{L^1(0,T;V_0^*)} \lesssim
    \norm[\big]{\beta_\lambda(X_\lambda)}_{L^1(0,T;V^*)} \lesssim
    \norm[\big]{\beta_\lambda(X_\lambda)}_{L^1(0,T;L^1(D))},
  \end{gather*}
  hence, again by Lemmata \ref{lm:stp} and \ref{lm:L1},
  $\norm{Y'_\lambda}_{L^1(0,T;V_0^*)}$ is bounded uniformly over
  $\lambda$. Moreover, since $B \cdot W \in L^2(\Omega;C([0,T];V_0))$
  and
  \[
  \norm[\big]{Y_\lambda}_{L^2(0,T;V)} \leq 
  \norm[\big]{X_\lambda}_{L^2(0,T;V)} + 
  \norm[\big]{B \cdot W}_{L^2(0,T;V)},
  \]
  we conclude that $(Y_\lambda)$ is bounded in $L^2(0,T;V)$. Simon's
  compactness criterion then implies that $Y_\lambda$, hence also
  $X_\lambda$, is relatively compact in $L^2(0,T;H)$. Since
  $X_{\lambda'} \wto X$ in $L^2(0,T;V)$, it follows that
  \[
  X_{\lambda'}(\omega,\cdot) \xrightarrow{\;\phantom{*}\;} X(\omega,\cdot) 
  \qquad \text{in } L^2(0,T;H),
  \]
  thus completing the proof.
\end{proof}

We are now going to show that the couple $(X,\xi)$ just constructed is
indeed the unique solution to the equation with ``smoothed'' noise
\eqref{eq:V0}.
\begin{proof}[Proof of Proposition \ref{prop:V0}]
  In spite of the above preparations, the argument is quite
  long, so we subdivide it into several steps.
  \smallskip\par\noindent
\textsc{Step 1.} In the notation of Proposition \ref{prop:cvg}, let
  $\omega \in \Omega'$ be arbitrary but fixed. Note that
  $X_{\lambda'} \to X$ in $L^2(0,T;H)$ implies that, passing to a
  further subsequence of $\lambda'$, denoted with the same symbol for
  simplicity, $X_{\lambda'}(t) \to X(t)$ in $H$ for almost all
  $t \in [0,T]$. Moreover, $X_{\lambda'} \wto X$ in $L^2(0,T;V)$
  implies that 
  \[
  \int_0^t AX_\lambda(s)\,ds \xrightharpoonup{\;\;\;} \int_0^t AX(s)\,ds
  \qquad \text{in } V^*
  \]
  for all $t \in [0,T]$. In fact, taking $\phi_0 \in V$ and
  $\phi:=s \mapsto 1_{[0,t]}(s) \phi_0 \in L^2(0,t;V)$, one obviously has
  $A^*\phi \in L^2(0,t;V^*)$ and
  \begin{align*}
    \int_0^t \ip{AX_\lambda(s)}{\phi_0}\,ds &=
    \int_0^T \ip{AX_\lambda(s)}{\phi(s)}\,ds =
    \int_0^T \ip{X_\lambda(s)}{A^*\phi(s)}\,ds\\
    &\quad \longrightarrow \int_0^T \ip{X(s)}{A^*\phi(s)}\,ds
    = \int_0^t \ip{AX(s)}{\phi_0}\,ds.
  \end{align*}
  Similarly, $\beta_{\lambda'}(X_{\lambda'}) \wto \xi$ in
  $L^1((0,T) \times D)$ implies
  \[
  \int_0^t \beta_{\lambda'}(X_{\lambda'}(s))\,ds \xrightharpoonup{\;\;\;}
  \int_0^t \xi(s)\,ds \qquad \text{in } L^1(D)
  \]
  for all $t \in [0,T]$. In particular, passing to the limit as
  $\lambda' \to 0$ in the regularized equation
  \eqref{eq:reg} yields
  \[
  X(t) + \int_0^t AX(s)\,ds + \int_0^t \xi(s)\,ds = X_0 + B \cdot W(t)
  \qquad \text{in $V_0^*$ for a.a. $t\in[0,T]$.}
  \]
  Since $AX \in L^2(0,T;V^*) \embed L^1(0,T;V_0^*)$ and $\xi \in
  L^1(0,T;L^1(D)) \embed L^1(0,T;V_0^*)$, recalling that $B \cdot W
  \in C([0,T];V_0)$, we infer that $X \in C([0,T];V_0^*)$, hence the
  previous identity is true for all $t \in [0,T]$. Moreover, it
  follows from $X \in L^\infty(0,T;H)$ that $X \in C_w([0,T];H)$,
  thanks Lemma~\ref{lm:Strauss}.  Note also that all terms expect the second
  one on the left-hand side take values in $L^1(D)$, and all terms
  except the third one on the left-hand side take values in $V^*$,
  hence the above identity holds true also in $L^1(D) \cap V^*$.

  Let us now show that $\xi \in \beta(X)$ a.e. in $(0,T) \times D$:
  $X_{\lambda'} \to X$ in $L^2(0,T;H)$ implies that, passing to a
  subsequence of $\lambda'$, still denoted by the same symbol,
  $X_{\lambda'} \to X$ a.e. in $(0,T) \times D$, hence also
  $(I+\lambda' \beta)^{-1} X_{\lambda'} \to X$ a.e. in
  $(0,T) \times D$. Since
  $\beta_{\lambda'}(X_{\lambda'}) \in \beta((I+\lambda'\beta)^{-1}
  X_{\lambda'})$ a.e. in $(0,T) \times D$ and
  $\beta_{\lambda'}(X_{\lambda'}) (I+\lambda'\beta)^{-1} X_{\lambda'}$ is
  bounded in $L^1((0,T) \times D)$ by \eqref{eq:plutone}, Br\'ezis'
  Lemma \ref{lm:Brezis} implies the claim. These relations and the
  weak convergence $\beta_{\lambda'}(X_{\lambda'}) \wto \xi$ in
  $L^1((0,T) \times D)$ also imply, by the weak lower semicontinuity
  of convex integrals, that
  \begin{align*}
  \int_0^T\!\!\int_D \bigl( j(X)+j^*(\xi) \bigr) &\leq 
  \liminf_{\lambda' \to 0} \int_0^T\!\!\int_D
  \bigl( j((I+\lambda' A)^{-1} X_{\lambda'})
   + j^*(\beta_{\lambda'}(X_{\lambda'})) \bigr)\\
  &= \liminf_{\lambda' \to 0} \int_0^T\!\!\int_D
  \beta_{\lambda'}(X_{\lambda'}) (I+\lambda' A)^{-1} X_{\lambda'} \leq N,
  \end{align*}
  where $N$ is a constant that depends on $\omega$.
  \smallskip\par\noindent
  \textsc{Step 2.} Still keeping $\omega$ fixed as in the previous
  step, we are going to show that the limits $X$ and $\xi$ constructed
  above are unique. Suppose there
  exist $(X_i,\xi_i)$, $\xi_i \in \beta(X_i)$ a.e. in
  $(0,T) \times D$, $i=1,2$, such that
  \[
  X_i(t) + \int_0^t AX_i(s)\,ds + \int_0^t \xi_i(s)\,ds = X_0 + B
  \cdot W(t)
  \]
  in $L^1(D) \cap V^*$ for all $t\in [0,T]$. Setting $X=X_1-X_2$ and
  $\xi=\xi_1-\xi_2$, it is enough to show that
  \begin{equation}  \label{eq:secca}
  X(t) + \int_0^t AX(s)\,ds + \int_0^t \xi(s)\,ds = 0
  \end{equation}
  in $L^1(D) \cap V^*$ for all $t\in [0,T]$ implies $X=0$ and
  $\xi=0$. By the hypotheses on $A$, there exists $m\in\mathbb{N}$
  such that $(I+\delta A)^{-m}$ maps $L^1(D)$ in
  $L^\infty(D)$. Therefore, setting
  \[
  X^\delta := (I+\delta A)^{-m}X, \qquad
  \xi^\delta := (I+\delta A)^{-m}\xi,
  \]
  one has
  \[
  X^\delta(t) + \int_0^t AX^\delta(s)\,ds + \int_0^t \xi^\delta(s)\,ds = 0
  \]
  for all $t \in [0,T]$, for which It\^o's formula and monotonicity of
  $A$ yield
  \[
  \frac12 \norm[\big]{X^\delta(t)}^2_H 
  + \int_0^t\!\!\int_D \xi^\delta(s,x)X^\delta(s,x)\,dx\,ds \leq 0.
  \]
  We can now take the limit as $\delta \to 0$. Since
  $(I+\delta A)^{-m}$ converges, in the strong operator topology, to
  the identity in $\cL(H)$, one has
  $\norm{X^\delta(t)}_H \to \norm{X(t)}_H$ for all $t \in [0,T]$.
  Passing to a subsequence of $\delta$, still denoted by the same
  symbol, we also have $X^\delta \to X$ and $\xi^\delta \to \xi$
  a.e. in $(0,T) \times D$, hence $X^\delta \xi^\delta \to X\xi$
  a.e. in $(0,T) \times D$. Let us show that $(X^\delta\xi^\delta)$ is
  uniformly integrable: by the symmetry of $j$ and $j^*$, and the
  abstract Jensen inequality of Lemma \ref{lm:J}, we have
  \[
  \abs{X_\delta \xi_\delta} \leq j(X_\delta) + j^*(\xi_\delta)
  \leq 
  (I+\delta A)^{-m} \bigl( j(X)+j^*(\xi) \bigr),
  \]
  where the term on the right-hand side converges to $j(X)+j^*(\xi)$
  in $L^1((0,T) \times D)$ as $\delta \to 0$, hence
  $(X^\delta\xi^\delta)$ is indeed uniformly integrable on
  $(0,T) \times D$. It follows by Vitali's convergence theorem that,
  for any $t\in [0,T]$,
  \[
  \int_0^t\!\!\int_D X^\delta \xi^\delta \to
  \int_0^t\!\!\int_D X\xi,
  \]
  hence also
  \[
  \frac12 \norm[\big]{X(t)}^2_H 
  + \int_0^t\!\!\int_D X(s,x)\xi(s,x)\,dx\,ds \leq 0.
  \]
  The monotonicity of $\beta$ immediately implies that $X(t)=0$ for
  all $t\in [0,T]$. Substituing in \eqref{eq:secca}, we are left with
  $\int_0^t\xi(s)\,ds=0$ in $L^1(D)$ for all $t\in[0,T]$, so that also
  $\xi=0$, and uniqueness is proved.
  \smallskip\par\noindent
  \textsc{Step 3.} The solution $(X,\xi)$ does \emph{not} have, a
  priori, any measurability in $\omega$, because of the way it has
  been constructed. We are going to show that in fact $X$ and $\xi$
  are predictable processes. The reasoning for $X$ is simple: with
  $\omega$ fixed, we have proved that from any subsequence of
  $\lambda$ one can extract a further subsequence $\lambda'$,
  depending on $\omega$, such that the convergences of Proposition
  \ref{prop:cvg} take place, and the limit $(X,\xi)$ is unique. This
  implies, by a well-known criterion of classical analysis, that the
  same convergences hold along the original sequence $\lambda$, which
  does \emph{not} depend on $\omega$. The convergence of
  $X_\lambda(\omega,\cdot)$ to $X(\omega,\cdot)$ in $L^2(0,T;H)$
  implies that \luca{$X:\Omega\to L^2(0,T;H)$ is measurable and}
  $X_\lambda(\omega,t)$ converges to $X(\omega,t)$ in $H$ in $\P
  \otimes dt$-measure, hence $X_{\bar\lambda}(\omega,t) \to
  X(\omega,t)$ in $H$ $\P \otimes dt$-a.e. along a subsequence
  $\bar\lambda$ of $\lambda$.  Since $X_\lambda$ is predictable, being
  adapted with continuous trajectories in $H$, we infer that $X$ is
  predictable.  Unfortunately a similar reasoning does \emph{not} work
  for $\xi$, because
  $\xi_\lambda(\omega):=\beta_\lambda(X_\lambda(\omega))$ converges
  only weakly in $L^1((0,T)\times D)$ for
  $\P$-a.a.~$\omega\in\Omega$.\footnote{%
    One may indeed deduce, using Mazur's lemma, that there exists, for
    each $\omega$ in a set of probability one, a sequence
    $(\tilde{\xi}_{\mu(\omega)}(\omega))_{\mu(\omega)}$ in the convex
    envelope of $(\xi_\lambda(\omega))_\lambda$ that converges to
    $\xi(\omega)$. However, the map $\omega \mapsto
    \tilde{\xi}_{\mu(\omega)}(\omega)$ needs not be measurable, hence
    we cannot infer measurability of its limit $\xi$.} %
  We shall prove instead that a subsequence of
  $\xi_\lambda:=\beta_\lambda(X_\lambda)$ converges weakly to $\xi$ in
  $L^1(\Omega \times (0,T) \times D)$. In fact, let $g \in
  L^\infty((0,T) \times D)$ be arbitrary but fixed. Then, setting
  \[
  F_\lambda(\omega) := \int_0^T\!\!\int_D \xi_\lambda(\omega,s,x)g(s,x)\,dx\,ds,
  \qquad
  F(\omega) := \int_0^T\!\!\int_D \xi(\omega,s,x)g(s,x)\,dx\,ds,
  \]
  we have $F_\lambda \to F$ in probability, and we claim that
  $F_\lambda \to F$ weakly in $L^1(\Omega)$. Let $h \in L^\infty(\Omega)$ be
  arbitrary but fixed, and introduce the even convex function
  \[
  j_0 := j^*(\cdot/M), \qquad M :=
  \frac{1}{\bigl(\norm{g}_{L^\infty((0,T)\times D)} \vee 1 \bigr)%
           \bigl(\norm{h}_{L^\infty(\Omega)} \vee 1 \bigr)}.
  \]
  Then, by Jensen's inequality,
  \begin{align*}
    \E j_0(F_\lambda h) &= \E j_0\biggl( \int_0^T\!\!\int_D \xi_\lambda(\omega,s,x)%
    g(s,x)h(\omega)\,dx\,ds \biggr)\\
    &\lesssim_{T,\abs{D}} \E\int_0^T\!\!\int_D j_0\bigl( \xi_\lambda(\omega,s,x)%
    g(s,x)h(\omega) \bigr)\,dx\,ds\\
    &\leq \E\int_0^T\!\!\int_D j^*\bigl( \xi_\lambda(\omega,s,x) \bigr)\,dx\,ds,
  \end{align*}
  where the last term is bounded by a constant independent of
  $\lambda$, as proved in Lemma \ref{lm:L1}. Since $j_0$ inherits the
  superlinearity at infinity of $j^*$, the criterion of de la Vall\'ee
  Poussin implies that $F_\lambda h$ is uniformly integrable, hence,
  since $F_\lambda h \to Fh$ in probability, that $F_\lambda h \to Fh$
  strongly in $L^1(\Omega)$ by Vitali's theorem. As $h$ was arbitrary,
  this implies that $F_\lambda \to F$ weakly in $L^1(\Omega)$, thus
  also that $\xi_\lambda \to \xi$ weakly in
  $L^1(\Omega \times (0,T) \times D)$ by arbitrariness of $g$. By the
  canonical identification of $L^1(\Omega \times (0,T) \times D)$ with
  $L^1(\Omega \times (0,T);L^1(D))$ and Mazur's lemma (see, e.g.,
  \cite[7), p.~360]{Bbk:EVT}), there exists a sequence
  $(\zeta_n)_{n\in\mathbb{N}}$ of convex combinations of
  $(\xi_\lambda)$ that converges strongly to $\xi$ in $L^1(D)$ in
  $\P \otimes dt$-measure, hence $\P \otimes dt$-a.e.  passing to a
  subsequence of $n$. Since $\xi_\lambda$, hence $\zeta_n$, are
  predictable for all $\lambda$ and $n$, respectively, it follows that
  $\xi$ is a predictable $L^1(D)$-valued process \luca{and $\xi:\Omega\to L^1((0,T)\times D))$
  is measurable.}
  \luca{Moreover, since $X_\lambda(\omega,\cdot)\to X(\omega,\cdot)$ in $L^2(0,T; H)$
  for $\P$-a.a. $\omega$ and $(X_\lambda)_\lambda$ is bounded in $L^2(\Omega; L^2(0,T;V))$,
  it follows that $X_\lambda\wto X$ in $L^2(\Omega; L^2(0,T; V))$. Therefore, an entirely analogous argument
  based on Mazur's lemma yields that $X:\Omega\to L^2(0,T;V)$ is measurable.}
  \smallskip\par\noindent
  \textsc{Step 4.} As last step, we are going to show that $X$ and
  $\xi$ satisfy also estimates in expectation. In particular, the weak
  and weak* lower semicontinuity of the norm ensures that, for $\P$-almost
  all $\omega \in \Omega$,
  \begin{align*}
  \norm[\big]{X(\omega,\cdot)}_{L^2(0,T;V)} &\leq \liminf_{\lambda \to 0}
  \norm[\big]{X_{\lambda}(\omega,\cdot)}_{L^2(0,T;V)},\\
  \norm[\big]{X(\omega,\cdot)}_{L^\infty(0,T;H)} &\leq 
  \liminf_{\lambda \to 0} \norm[\big]{X_{\lambda}(\omega,\cdot)}_{L^\infty(0,T;H)},\\
  \norm[\big]{\xi(\omega,\cdot)}_{L^1(Q)} &\leq \liminf_{\lambda \to 0}
  \norm[\big]{\beta_{\lambda}(X_{\lambda}(\omega,\cdot))}_{L^1(Q)}.
  \end{align*}
  Taking expectations and recalling Lemmata \ref{lm:aspetta} and
  \ref{lm:L1}, it follows by Fatou's lemma that, for a constant $N$,
  \begin{align*}
  \E \norm[\big]{X}_{L^2(0,T;V)}^2 &\leq 
  \E \Bigl( \liminf_{\lambda \to 0} \norm[\big]{X_{\lambda}}^2_{L^2(0,T;V)}
     \Bigr) \leq \liminf_{\lambda \to 0}
  \E \norm[\big]{X_{\lambda}}^2_{L^2(0,T;V)} < N,\\
  \E \norm[\big]{X}_{L^\infty(0,T;H)}^2 &\leq 
  \E \Bigl( \liminf_{\lambda \to 0} \norm[\big]{X_{\lambda}}^2_{L^\infty(0,T;H)}
     \Bigr) \leq \liminf_{\lambda \to 0}
  \E \norm[\big]{X_{\lambda}}^2_{L^\infty(0,T;H)} < N,\\
  \E \norm[\big]{\xi}_{L^1(0,T;L^1(D))} &\leq 
  \E \Bigl( \liminf_{\lambda \to 0} \norm[\big]{\xi_{\lambda}}_{L^1(0,T;L^1(D))}
     \Bigr) \leq \liminf_{\lambda \to 0}
  \E \norm[\big]{\xi_\lambda}_{L^1(0,T;L^1(D))} < N,
  \end{align*}
  i.e.
  \begin{align*}
  X   &\in L^2(\Omega;L^\infty(0,T;H)) \cap L^2(\Omega;L^2(0,T;V)),\\
  \xi &\in L^1(\Omega \times(0,T) \times D).
  \end{align*}
  The proof is thus complete.
\end{proof}

We conclude this section with a corollary that will be used in the
following.
\begin{coroll}   \label{cor:ona}
  There exists a constant $N$ such that
  \[
  \E\int_0^T\!\!\int_D \bigl( j(X) + j^*(\xi) \bigr) < 
  N\Bigl( \norm[\big]{X_0}_{L^2(\Omega;H)}^2
  + \norm[\big]{B}^2_{L^2(\Omega;L^2(0,T;\cL^2(U,H)))} \Bigr).
  \]
\end{coroll}
\begin{proof}
  Thanks to Step 3 in the previous proof, there exists a sequence
  $\lambda$, independent of $\omega$, such that $X_\lambda \to X$
  a.e. in $(0,T) \times D$ and $\beta_\lambda(X_\lambda) \to \xi$
  weakly in $L^1((0,T) \times D)$. Proceeding as in the first part of
  the proof of Lemma \ref{lm:L1}, Lemma \ref{lm:aspetta} implies that
  there exists a constant $N$ such that
  \[
  \E\int_0^T\!\!\int_D \bigl( j(I+\lambda\beta)^{-1}X_\lambda) +
  j^*(\beta_\lambda(X_\lambda)) \bigr) \,dx\,ds < \bar{N}(X_0,B),
  \]
  where
  $\bar{N}(X_0,B):=N\bigl( \norm{X_0}_{L^2(\Omega;H)}^2 +
  \norm{B}^2_{L^2(\Omega;L^2(0,T;\cL^2(U,H)))} \bigr)$.
  Therefore, in analogy to Step 4 of the previous proof, two
  applications of Fatou's lemma yield
  \[
  \E\int_0^T\!\!\int_D j(X) \leq \liminf_{\lambda \to 0}
  \E\int_0^T\!\!\int_D j((I+\lambda\beta)^{-1}X_\lambda) < \bar{N}(X_0,B),
  \]
  as well as, by the weak lower semicontinuity of convex integrals and
  Fatou's lemma again,
  \[
  \E\int_0^T\!\!\int_D j^*(\xi) \leq \liminf_{\lambda \to 0}
  \E\int_0^T\!\!\int_D j^*(\beta_\lambda(X_\lambda)) < \bar{N}(X_0,B).
  \qedhere
  \]
\end{proof}


\ifbozza\newpage\else\fi
\section{Well-posedness with additive noise}
\label{sec:add}
In this section we prove well-posedness for the equation
\begin{equation}
\label{eq:add}
dX(t) + AX(t)\,dt + \beta(X(t))\,dt \ni B(t)\,dW(t),
\qquad X(0)=X_0,
\end{equation}
where $B$ is an $\cL^2(U,H)$-valued process. Note that this is just
equation \eqref{eq:0} with additive noise.

\begin{prop}  \label{prop:add}
  Assume that $X_0\in L^2(\Omega,\cF_0,\P; H)$ and that
  \[
  B \in L^2(\Omega;L^2(0,T;\cL^2(U,H)))
  \]
  is measurable and adapted.  Then equation \eqref{eq:V0}
  is well posed in $\mathscr{J}$.
  Moreover, $X(\omega,\cdot) \in C_w([0,T];H)$ for $\P$-almost all
  $\omega \in \Omega$.
\end{prop}
\begin{proof}
  We shall proceed in several steps: first we approximate the
  coefficient $B$ in such a way that the corresponding equation can be
  uniquely solved by the methods of the previous section. Then we pass
  to the limit in an appropriate way, obtaining a solution to
  \eqref{eq:add}, which is then shown to be unique.
  \smallskip\par\noindent
  \textsc{Step 1.} By Assumption A(iv), there exists
  $m \in \mathbb{N}$ such that $(I+A)^{-m}$ maps continuously
  $L^1$ to $L^\infty$. The space $V_0:=\dom(A^m)$, endowed with inner
  product
  \[
  \ip{u}{v}_{V_0} := \ip{u}{v}_H + \ip{A^mu}{A^mv}_H, \qquad u,v \in
  \dom(A^m),
  \]
  is a Hilbert space densely and continuously embedded in
  $V$. Moreover, the diagram
  \[
  \dom(A^m) \xrightarrow{\;(I+A)^m\;} L^1(D)
  \xrightarrow{\;(I+A)^{-m}\;} L^\infty(D)
  \]
  immediately shows that $V_0$ is also continuously embedded in
  $L^\infty(D)$. In particular, all hypotheses on $V_0$ of the
  previous section are met. Moreover, by the ideal property of
  Hilbert-Schmidt operators, setting, for any $\varepsilon>0$,
  \[
  B^\varepsilon := (I + \varepsilon A)^{-m} B,
  \]
  we have $B^\varepsilon \in L^2(\Omega;L^2(0,T;\cL^2(U,V_0)))$.  Then
  it follows by Proposition \ref{prop:V0} that, for any
  $\varepsilon>0$, there exist predictable processes
  \begin{gather*}
  X^\varepsilon \in L^2(\Omega;L^\infty(0,T;H)) 
     \cap L^2(\Omega;L^2(0,T;V)),\\
  \xi^\varepsilon \in L^1(\Omega \times(0,T) \times D),
  \end{gather*}
  with $X^\varepsilon(\omega,\cdot) \in C_w([0,T];H)$ for $\P$-almost
  all $\omega \in \Omega$, such that
  \begin{equation}  \label{eq:eps}
  X^\varepsilon(t) + \int_0^t AX^\varepsilon(s)\,ds
  + \int_0^t \xi^\varepsilon(s)\,ds = X_0 + \int_0^t B^\varepsilon(s)\,dW(s)
  \end{equation}
  in $V^* \cap L^1(D)$ for all $t \in [0,T]$. Moreover,
  $\xi^\varepsilon \in \beta(X^\varepsilon)$ a.e. in $(0,T) \times D$
  and
  $j(X^\varepsilon) + j^*(\xi^\varepsilon) \in L^1((0,T) \times D)$
  $\P$-almost surely.
  \smallskip\par\noindent
  \textsc{Step 2.} For any $\varepsilon>0$, the equation in $V^*$
  \[
  X^\varepsilon_\lambda(t) + \int_0^t AX^\varepsilon_\lambda(s)\,ds
  + \int_0 \beta_\lambda(X^\varepsilon_\lambda(s))\,ds
  = X_0 + \int_0^t B^\varepsilon(s)\,dW(s)
  \]
  admits a unique (variational) strong solution
  $X_\lambda^\varepsilon$. Taking into account the coercivity of $A$
  and the monotonicity of $\beta_\lambda$, It\^o's formula yields, for
  any $\delta>0$,
  \begin{align*}
  &\norm[\big]{X^\varepsilon_\lambda(t) - X^\delta_\lambda(t)}_H^2
  + \int_0^t \norm[\big]{X^\varepsilon_\lambda(s) - X^\delta_\lambda(s)}_V^2\,ds\\
  &\hspace{3em} \lesssim \int_0^t
  \bigl( X^\varepsilon_\lambda(s) - X^\delta_\lambda(s) \bigr)
  \bigl( B^\varepsilon(s) - B^\delta(s) \bigr)\,dW(s)
  + \int_0^t \norm[\big]{B^\varepsilon(s) - B^\delta(s)}^2_{\cL^2(U,H)}\,ds.
  \end{align*}
  Taking supremum in time and expectation, it easily follows from
  Lemma \ref{lm:DY} that
  \begin{align*}
  &\norm[\big]{X^\varepsilon_\lambda - X^\delta_\lambda}_{L^2(\Omega;L^\infty(0,T;H))}
  + \norm[\big]{X^\varepsilon_\lambda - X^\delta_\lambda}_{L^2(\Omega;L^2(0,T;V))}\\
  &\hspace{5em} \lesssim
     \norm[\big]{B^\varepsilon - B^\delta}_{L^2(\Omega;L^2(0,T;\cL^2(U,H)))}.
  \end{align*}
  On the other hand, the proof of Proposition \ref{prop:V0} shows that
  there exists a sequence $\lambda$, independent of $\varepsilon$,
  such that, for $\P$-almost all $\omega \in \Omega$,
  \begin{align*}
    X^\varepsilon_\lambda(\omega,\cdot) &\xrightharpoonup{\;*\;} 
    X^\varepsilon(\omega,\cdot)
    & &\text{in } L^\infty(0,T;H),\\
    X^\varepsilon_\lambda(\omega,\cdot) &\xrightharpoonup{\;\phantom{*}\;}
    X^\varepsilon(\omega,\cdot)
    & &\text{in } L^2(0,T;V),\\
    \beta_\lambda(X^\varepsilon_\lambda(\omega,\cdot)) 
    &\xrightharpoonup{\;\phantom{*}\;} \xi^\varepsilon(\omega,\cdot)
    & &\text{in } L^1((0,T) \times D)
  \end{align*}
  as $\lambda \to 0$. Since the weak* limit in $L^\infty(0,T;H)$ as
  $\lambda \to 0$ of $X^\varepsilon_\lambda-X^\delta_\lambda$ is
  $X^\varepsilon-X^\delta$, the weak* lower semicontinuity of the norm
  implies
  \[
  \norm[\big]{X^\varepsilon-X^\delta}_{L^\infty(0,T;H)} \leq
  \liminf_{\lambda \to 0}
  \norm[\big]{X_\lambda^\varepsilon-X_\lambda^\delta}_{L^\infty(0,T;H)},
  \]
  thus also, by Fatou's lemma,
  \[
  \E\norm[\big]{X^\varepsilon-X^\delta}^2_{L^\infty(0,T;H)} \leq
  \E \liminf_{\lambda \to 0}
  \norm[\big]{X_\lambda^\varepsilon-X_\lambda^\delta}^2_{L^\infty(0,T;H)}
  \lesssim \E \norm[\big]{B^\varepsilon - B^\delta}^2_{L^2(0,T;\cL^2(U,H))}.
  \]
  An entirely similar argument yields
  \[
  \E\norm[\big]{X^\varepsilon-X^\delta}^2_{L^2(0,T;V)}
  \lesssim \E \norm[\big]{B^\varepsilon - B^\delta}^2_{L^2(0,T;\cL^2(U,H))},
  \]
  so that
  \begin{align*}
  &\norm[\big]{X^\varepsilon - X^\delta}_{L^2(\Omega;L^\infty(0,T;H))}
  + \norm[\big]{X^\varepsilon - X^\delta}_{L^2(\Omega;L^2(0,T;V))}\\
  &\hspace{5em} \lesssim
     \norm[\big]{B^\varepsilon - B^\delta}_{L^2(\Omega;L^2(0,T;\cL^2(U,H)))}.
  \end{align*}
  Taking into account that
  $\norm[\big]{B^\varepsilon - B}_{L^2(\Omega;L^2(0,T;\cL^2(U,H)))}
  \to 0$
  as $\varepsilon \to 0$, it follows that $(X^\varepsilon)$ is a
  Cauchy sequence in
  $E:=L^2(\Omega;L^\infty(0,T;H)) \cap L^2(\Omega;L^2(0,T;V))$, hence
  there exists $X \in E$ such that $X^\varepsilon$ converges
  (strongly) to $X$ in $E$ as $\varepsilon \to 0$. In particular, the
  limit process $X$ is predictable.
  Moreover, by Corollary \ref{cor:ona}, there exists a constant $N$
  such that
  \begin{equation}  \label{eq:steps}
    \begin{split}
    \E\int_0^T\!\!\int_D \bigl(j(X^\varepsilon) + j^*(\xi^\varepsilon)\bigr)
    \,dx\,ds &< N\Bigl( \norm[\big]{X_0}^2_{L^2(\Omega;H)} 
    + \norm[\big]{B^\varepsilon}^2_{L^2(\Omega;L^2(0,T;\cL^2(U,H)))} \Bigr)\\
    &\leq N\Bigl( \norm[\big]{X_0}^2_{L^2(\Omega;H)} 
    + \norm[\big]{B}^2_{L^2(\Omega;L^2(0,T;\cL^2(U,H)))} \Bigr),
    \end{split}
  \end{equation}
  as it follows by the ideal property of Hilbert-Schmidt operators and
  the contractivity of $(I+\varepsilon A)^{-1}$.
  The criterion by de la Vall\'ee Poussin then implies
  that $(\xi^\varepsilon)$ is uniformly integrable on
  $\Omega \times (0,T) \times D$, hence, by the Dunford-Pettis
  theorem, $(\xi^\varepsilon)$ is weakly relatively compact in
  $L^1(\Omega \times (0,T) \times D)$. Therefore, passing to a
  subsequence of $\varepsilon$, denoted by the same symbol, there
  exists $\xi$ belonging to the latter space such that
  $\xi^\varepsilon \to \xi$ \luca{therein} in the weak topology. In
  particular, by an argument based on Mazur's lemma, entirely
  analogous to that used in Step 3 of the proof of Proposition
  \ref{prop:V0}, one infers that $\xi$ is a predictable process.
  \smallskip\par\noindent
  \textsc{Step 3.} We can now pass to the limit as $\varepsilon \to 0$
  in \luca{Equation} \eqref{eq:eps}, by a reasoning analogous to the one use
  in Step 1 of the proof of Proposition \ref{prop:V0}. As proved in
  the previous step, $X^\varepsilon$ converges strongly to $X$ in
  $L^2(\Omega;L^\infty(0,T;H))$, hence
  \[
  \operatorname*{ess\,sup}_{t\in[0,T]} \norm[\big]{X^\varepsilon(t)-X(t)}_H 
  \to 0
  \]
  in probability as $\varepsilon \to 0$.
  Let $\phi_0 \in V_0$ be arbitrary. Since $V_0 \embed L^\infty(D)$, one has
  \[
  \ip[\big]{X^\varepsilon(t)}{\phi_0} \to \ip[\big]{X(t)}{\phi_0}
  \]
  in probability for almost all $t \in [0,T]$. Let us set, for an
  arbitrary but fixed $t \in [0,T]$,
  $\phi:s \mapsto 1_{[0,t]}(s) \phi_0 \in L^2(0,T;V)$, so that
  $A\phi \in L^2(0,T;V^*)$. Recalling that $X^\varepsilon \to X$
  (strongly, hence also weakly) in $L^2(\Omega;L^2(0,T;V))$, it
  follows immediately that $X^\varepsilon \wto X$ in $L^2(0,T;V)$ in
  measure, hence
  \begin{align*}
    \int_0^t \ip{AX^\varepsilon}{\phi_0}\,ds &=
    \int_0^T \ip{AX^\varepsilon(s)}{\phi(s)}\,ds =
    \int_0^T \ip{X^\varepsilon(s)}{A\phi(s)}\,ds\\
    &\quad \to \int_0^T \ip{X(s)}{A\phi(s)}\,ds
    = \int_0^t \ip{AX(s)}{\phi_0}\,ds
  \end{align*}
  in probability as $\varepsilon \to 0$. A completely analogous
  reasoning shows that
  \[
  \int_0^t \ip{\xi^\varepsilon(s)}{\phi_0}\,ds \to \int_0^t
  \ip{\xi(s)}{\phi_0}\,ds
  \]
  in probability as $\varepsilon \to 0$. Doob's maximal inequality
  and the convergence
  \[
  \norm[\big]{B^\varepsilon - B}_{L^2(\Omega;L^2(0,T;\cL^2(U,H)))}
  \xrightarrow{\;\varepsilon \to 0\;} 0
  \]
  readily yield also that $B^\varepsilon \cdot W(t) \to B \cdot W(t)$
  in $H$ in probability for all $t \in [0,T]$. In particular, since
  $\phi_0\in V_0$ and $t \in [0,T]$ are arbitrary, we infer that
  \[
  X(t) + \int_0^t AX(s)\,ds + \int_0^t \xi(s)\,ds = 
  X_0 + \int_0^t B(s)\,dW(s)
  \]
  holds in $V_0^*$ for almost all $t$. Recalling that
  $\xi \in L^1(0,T;L^1(D)) \embed L^1(0,T;V_0^*)$, so that all terms
  except the first on the left-hand side have trajectories in
  $C([0,T];V_0^*)$, we conclude that the identity holds for all
  $t \in [0,T]$. Moreover, thanks to Lemma \ref{lm:Strauss},
  $X \in C([0,T];V_0^*)$ and $X \in L^\infty(0,T;H)$ imply
  $X \in C_w([0,T];H)$. Note also that all terms bar the second
  [third] one on the left-hand side are $L^1(D)$-valued
  [$V^*$-valued], hence the identity holds in $L^1(D) \cap V^*$ for
  all $t \in [0,T]$.
  \smallskip\par\noindent
  \textsc{Step 4.} Convergence of $X^\varepsilon \to X$ in
  $L^2(\Omega;L^\infty(0,T;H))$ implies convergence in measure in
  $\Omega \times (0,T) \times D$, hence, by Fatou's lemma,
  \eqref{eq:steps} yields
  \[
  \E\int_0^T\!\!\int_D j(X) < \bar{N}(X_0,B),
  \]
  where $\bar{N}(X_0,B)$ is the constant appearing in the last term of
  \eqref{eq:steps}. Similarly, since $\xi^\varepsilon \to \xi$ weakly
  in $L^1(\Omega \times (0,T) \times D)$, \eqref{eq:steps} and the
  weak lower semicontinuity of convex integrals yield
  \[
  \E\int_0^T\!\!\int_D j^*(\xi) < \bar{N}(X_0,B).
  \]
  To complete the proof of existence, we only need to show that
  $\xi \in \beta(X)$ a.e. in $\Omega \times (0,T) \times D$. Note that, passing
  to a subsequence of $\varepsilon$, still denoted by the same symbol,
  we have $X^\varepsilon \to X$ a.e. in $\Omega \times (0,T) \times D$. Recalling
  that $\xi^\varepsilon \in \beta(X^\varepsilon)$ a.e. in
  $\Omega \times (0,T) \times D$, \eqref{eq:steps} again implies
  \[
  \E\int_0^T\!\!\int_D X^\varepsilon \xi^\varepsilon = 
  \E\int_0^T\!\!\int_D \bigl(j(X^\varepsilon) + j^*(\xi^\varepsilon)\bigr)
  < \bar{N}(X_0,B).
  \]
  It follows by monotonicity that
  $X^\varepsilon\xi^\varepsilon \geq 0$, hence
  $X^\varepsilon\xi^\varepsilon \in L^1(\Omega \times (0,T) \times
  D)$.
  Br\'ezis' Lemma \ref{lm:Brezis} then yields $\xi \in \beta(X)$ a.e. in
  $\Omega \times (0,T) \times D$.

  Uniqueness and continuous dependence of the solution on the initial
  datum is an immediate consequence of the next result.
\end{proof}

We first need to introduce weighted (in time) versions of some spaces
of processes.  For any $p\in [1,\infty]$ and $\alpha \geq 0$, we shall
denote by $L^p_\alpha(0,T)$ the space $L^p(0,T)$ endowed with the norm
$\norm{f}_{L^p_\alpha(0,T)}:=\norm{t\mapsto e^{-\alpha
    t}f(t)}_{L^p(0,T)}$.  It is clear that $L^p(0,T)$ and
$L^p_\alpha(0,T)$, for different values of $\alpha$, are all
isomorphic (their norms are equivalent). Completely similar notation
will be used for vector-valued $L^p$ and $L^p_\alpha$ spaces. For
typographical economy, restricted only to the formulation of the
following proposition, let us define the Banach space
\[
F_\alpha := L^2(\Omega;L^\infty_\alpha(0,T;H)) 
           \cap L^2(\Omega;L^2_\alpha(0,T;V)),
\]
endowed with the norm
\[
\norm{\cdot}_{F_\alpha} := 
\norm{\cdot}_{L^2(\Omega;L^\infty_\alpha(0,T;H)) \cap L^2(\Omega;L^2_\alpha(0,T;V))} +
\sqrt{\alpha} \norm{\cdot}_{L^2(\Omega;L^2_\alpha(0,T;H))}
\]
\begin{prop}   \label{prop:contra}
  Let $(X_1,\xi_1)$, $(X_2,\xi_2)\in\mathscr{J}$ be solutions to
  \eqref{eq:add} with initial values $X_{01}$,
  $X_{02}\in L^2(\Omega, \cF_0;H)$ and progressively measurable
  diffusion coefficients $B_1$,
  $B_2\in L^2(\Omega; L^2(0,T; \cL^2(U,H)))$, respectively.  Then, for
  any $\alpha\geq0$,
  \[
    \norm{X_1-X_2}_{F_\alpha}\lesssim \norm{X_{01}-X_{02}}_{L^2(\Omega;H)}+
    \norm{B_1-B_2}_{L^2(\Omega; L^2_\alpha(0,T; \cL^2(U,H)))}.
  \]
  In particular, there is a unique solution $(X,\xi)\in\mathscr{J}$ to \eqref{eq:add}.
\end{prop}
\begin{proof}
  Setting
  \[
  Y := X_1-X_2, \qquad Y_0 := X_{01}-X_{02}, \qquad G:=B_1-B_2,
  \]
  one has
  \[
  Y(t) + \int_0^t AY(s)\,ds + \int_0^t \zeta(s)\,ds
  = Y_0 + \int_0^t G(s)\,dW(s)
  \]
  in $V^* \cap L^1(D)$, where $\zeta:=\xi_1-\xi_2$, and $\xi_1$,
  $\xi_2$ are defined in the obvious way. By the hypotheses on $A$,
  there exists $m \in \mathbb{N}$ such that, using the notation
  $h^\delta:=(I+\delta A)^{-m}h$ for any $h$ for which it makes sense,
  \[
  AY^\delta, \; \zeta^\delta \in L^1(\Omega;L^1(0,T;H)),
  \]
  while $Y_0^\delta$ and $G^\delta$ have the same
  integrability properties of $Y$, $Y_0$ and $G$, respectively. In
  particular, we have
  \[
  Y^\delta(t) + \int_0^t AY^\delta(s)\,ds + \int_0^t \zeta^\delta(s)\,ds
  = Y^\delta_0 + \int_0^t G^\delta(s)\,dW(s)
  \]
  in $V^*$. Let $\alpha>0$ be arbitrary but fixed, and add a
  superscript $\alpha$ to any process that is multiplied pointwise by
  the function $t \mapsto e^{-\alpha t}$. The integration by parts
  formula yields
  \[
  Y^{\delta,\alpha}(t) + \int_0^t (A+\alpha I)Y^{\delta,\alpha}(s)\,ds +
  \int_0^t \zeta^{\delta,\alpha}(s)\,ds = Y^\delta_0 + \int_0^t
  G^{\delta,\alpha}(s)\,dW(s),
  \]
  to which we can apply It\^o's formula for the square of
  the norm in $H$, obtaining, using the coercivity of $A$,
  \begin{align*}
  &\norm[\big]{Y^{\delta,\alpha}(t)}_H^2
  + 2\alpha \int_0^t \norm[\big]{Y^{\delta,\alpha}(s)}_H^2\,ds
  + 2 C \int_0^t \norm[\big]{Y^{\delta,\alpha}(s)}^2_V\,ds\\
  &\hspace{3em} + 2\int_0^t \ip[\big]{Y^{\delta,\alpha}(s)}{\zeta^{\delta,\alpha}(s)}\,ds\\
  &\hspace{5em} \leq \norm[\big]{Y_0^\delta}_H^2
  + \int_0^t Y^{\delta,\alpha}(s) G^{\delta,\alpha}(s)\,dW(s)
  + \int_0^t \norm[\big]{G^{\delta,\alpha}(s)}_{\cL^2(U,H)}^2\,ds.
  \end{align*}
  We are now going to pass to the limit as $\delta \to 0$: the first
  term on the left-hand side and on the right-hand side clearly
  converge to $\norm{Y^\alpha(t)}_H^2$ and $\norm{Y_0}_H^2$,
  respectively. Since $(I+\delta A)^{-1}$ converges to the identity in
  $H$ as well as in $V$ in the strong operator topology, the dominated
  convergence theorem yields
  \begin{align*}
  \int_0^t \norm[\big]{Y^{\delta,\alpha}(s)}^2_V\,ds
  &\longrightarrow
  \int_0^t \norm[\big]{Y^{\alpha}(s)}^2_V\,ds,\\
  \int_0^t \norm[\big]{G^{\delta,\alpha}(s)}_{\cL^2(U,H)}^2\,ds
  &\longrightarrow \int_0^t \norm[\big]{G^\alpha(s)}_{\cL^2(U,H)}^2\,ds
  \end{align*}
  as $\delta \to 0$ for all $t \in [0,T]$.
  Defining the real local martingales
  \[
  M^{\delta,\alpha} := (Y^{\delta,\alpha} G^{\delta,\alpha}) \cdot W, \qquad
  M^\alpha := (Y^\alpha G^\alpha) \cdot W,
  \]
  in order to establish convergence in probability (uniformly on
  compact sets) of the sequence $M^{\delta,\alpha}$ to $M^\alpha$ as
  $\delta \to 0$, it is sufficient to show that
  $[M^{\delta,\alpha}-M^\alpha,M^{\delta,\alpha}-M^\alpha]_T$
  converges to zero in probability. To this purpose, note that
  \begin{align*}
  [M^{\delta,\alpha}-M^\alpha,M^{\delta,\alpha}-M^\alpha]^{1/2}_T &= 
  \norm[\big]{Y^{\delta,\alpha} G^{\delta,\alpha}%
              - Y^\alpha G^\alpha}_{L^2(0,T;\cL^2(U,\erre))}\\
  &\leq \norm[\big]{Y^{\delta,\alpha} G^{\delta,\alpha}%
                    - Y^{\delta,\alpha} G^{\alpha}}_{L^2(0,T;\cL^2(U,\erre))}\\
  &\quad + \norm[\big]{Y^{\delta,\alpha} G^{\alpha}%
                       - Y^\alpha G^\alpha}_{L^2(0,T;\cL^2(U,\erre))},
  \end{align*}
  where
  \[
  \norm[\big]{Y^{\delta,\alpha}(t) G^{\delta,\alpha}(t)
    - Y^{\delta,\alpha}(t) G^{\alpha}(t)}_{\cL^2(U,\erre))}
  \leq \norm[\big]{Y^\alpha(t)}_H
  \norm[\big]{G^{\delta,\alpha}(t) - G^\alpha(t)}_{\cL^2(U,H))}
  \]
  for all $t \in [0,T]$. Since the right-hand side converges to $0$ as
  $\delta \to 0$ and it is bounded by
  $2\norm{Y^\alpha}_{L^\infty(0,T;H)}
  \norm{G^\alpha(t)}_{\cL^2(U,H)}$, and
  $G^\alpha \in L^2(0,T;\cL^2(U,H))$, the dominated
  convergence theorem yields
  \[
  \norm[\big]{Y^{\delta,\alpha} G^{\delta,\alpha}
  - Y^{\delta,\alpha} G^{\alpha}}_{L^2(0,T;\cL^2(U,\erre))} \to 0
  \]
  as $\delta \to 0$. A completely analogous argument shows
  that
  $\norm[\big]{Y^{\delta,\alpha} G^{\alpha} - Y^\alpha
    G^\alpha}_{L^2(0,T;\cL^2(U,\erre))}$ tends to $0$ as
  $\delta \to 0$ as well.

  We are now going to show that
  $Y^{\delta,\alpha} \zeta^{\delta,\alpha} \to Y^\alpha \zeta^\alpha$ in
  $L^1(\Omega \times (0,T) \times D)$, which clearly implies that
  \[
  \int_0^t\!\!\int_D Y^{\delta,\alpha} \zeta^{\delta,\alpha} \to
  \int_0^t\!\!\int_D Y^\alpha \zeta^\alpha
  \]
  in probability for all $t \in [0,T]$. Since
  $Y^{\delta,\alpha} \to Y^\alpha$ and
  $\zeta^{\delta,\alpha} \to \zeta^\alpha$ in measure in
  $\Omega \times (0,T) \times D$, Vitali's theorem implies strong
  convergence in $L^1$ if the sequence
  $(Y^{\delta,\alpha} \zeta^{\delta,\alpha})$ is uniformly integrable
  in $\Omega \times (0,T) \times D$. In turn, the latter is certainly
  true if $\bigl(\abs{Y^{\delta,\alpha} \zeta^{\delta,\alpha}}\bigr)$
  is dominated by a sequence that converges strongly in $L^1$. In
  order to prove this property, note that $j$ and $j^*$ are increasing
  on $\erre_+$, hence
  \begin{align*}
  \frac14 \abs[\big]{Y^{\delta,\alpha}(\omega,t,x) \zeta^{\delta,\alpha}(\omega,t,x)}
  &\leq j\bigl( e^{-\alpha t} \abs{Y^\delta(\omega,t,x)}/2 \bigr) 
  + j^*\bigl( e^{-\alpha t} \abs{\zeta^\delta(\omega,t,x)}/2 \bigr)\\
  &\leq j\bigl( \abs{Y^\delta(\omega,t,x)}/2 \bigr) 
  + j^*\bigl( \abs{\zeta^\delta(\omega,t,x)}/2 \bigr),
  \end{align*}
  so that, by the symmetry of $j$ and $j^*$, and by the Jensen
  inequality of Lemma \ref{lm:J},
  \[
  \frac14 \abs[\big]{Y^{\delta,\alpha} \zeta^{\delta,\alpha}}
  \leq j(Y^\delta/2) + j^*(\zeta^\delta/2)
  \leq (I+\delta A)^{-m} \bigl( j(Y/2) + j^*(\zeta/2) \bigr),
  \]
  where, by convexity and symmetry,
  \[
  j(Y/2) = j\Bigl( \frac12 X_1 + \frac12(-X_2) \Bigr) 
  \leq \frac12 \bigl( j(X_1) + j(X_2) \bigr)
  \in L^1(\Omega \times (0,T) \times D),
  \]
  and, completely analogously,
  \[
  j^*(\zeta/2) \leq \frac12 \bigl( j^*(\xi_1) + j^*(\xi_2) \bigr)
  \in L^1(\Omega \times (0,T) \times D),
  \]
  hence
  \[
  \abs[\big]{Y^{\delta,\alpha} \zeta^{\delta,\alpha}} \lesssim
  (I+\delta A)^{-m} \bigl( j(X_1) + j(X_2) + j^*(\xi_1) + j^*(\xi_2) \bigr).
  \]
  Since the right-hand side of this expression converges strongly in
  $L^1(\Omega \times (0,T) \times D)$ as $\delta \to 0$, it is, a
  fortiori, uniformly integrable, and so is the left-hand side.

  We have thus obtained
  \begin{align*}
  &\norm[\big]{Y^\alpha(t)}_H^2
  + 2\alpha \int_0^t \norm[\big]{Y^\alpha(s)}_H^2\,ds
  + 2\int_0^t \cE\bigl(Y^\alpha(s),Y^\alpha(s)\bigr) \,ds\\
  &\hspace{3em}
  + 2\int_0^t\!\!\int_D Y^\alpha(s,x) \zeta^\alpha(s,x) \,dx\,ds\\
  &\hspace{5em} \leq \norm[\big]{Y_0}_H^2
  + \int_0^t Y^\alpha(s) G^\alpha(s)\,dW(s)
  + \int_0^t \norm[\big]{G^\alpha(s)}_{\cL^2(U,H)}^2\,ds,
  \end{align*}
  where, by monotonicity,
  $Y^\alpha \zeta^\alpha = e^{-2\alpha \cdot} (X_1-X_2)(\xi_2-\xi_2)
  \geq 0$,
  hence, taking the $L^\infty(0,T)$ norm and expectation on both
  sides,
  \begin{multline*}
  \norm[\big]{Y^\alpha}_{L^2(\Omega;L^\infty(0,T;H))} +
  \sqrt{\alpha} \norm[\big]{Y^\alpha}_{L^2(\Omega;L^2(0,T;H))} +
  \norm[\big]{Y^\alpha}_{L^2(\Omega;L^2(0,T;V))}\\
  \lesssim
  \norm[\big]{Y_0}_{L^2(\Omega;H)} + \biggl( \E\sup_{t\leq T} \abs[\bigg]{%
  \int_0^t Y^\alpha(s) G^\alpha(s)\,dW(s)} \biggr)^{1/2} 
  + \norm[\big]{G^\alpha}_{L^2(\Omega;L^2(0,T;\cL^2(U,H)))}.
  \end{multline*}
  By Lemma \ref{lm:DY}, one has
  \begin{align*}
  \biggl( \E\sup_{t\leq T} \abs[\bigg]{\int_0^t Y^\alpha(s) G^\alpha(s)\,dW(s)}
  \biggr)^{1/2} &\leq \varepsilon \norm[\big]{Y^\alpha}_{L^2(\Omega;L^\infty(0,T;H))}\\
  &\quad + N(\varepsilon) \norm[\big]{G^\alpha}_{L^2(\Omega;L^2(0,T;\cL^2(U,H)))},
  \end{align*}
  with $\varepsilon>0$ arbitrary. Choosing $\varepsilon$ sufficiently
  small and rearranging terms, one obtains
  \[
    \norm{X_1-X_2}_{F_\alpha}\lesssim \norm{X_{01}-X_{02}}_{L^2(\Omega;H)}+
    \norm{B_1-B_2}_{L^2(\Omega; L^2_\alpha(0,T; \cL^2(U,H)))}
  \]
  as claimed.
  
  Choosing $\alpha=0$, $X_{01}=X_{02}$, and $B_1=B_2$, one gets
  immediately $X_1=X_2$, hence also, by substitution,
  \[
  \int_0^t(\xi_1(s)-\xi_2(s))\,ds = 0 \qquad\forall t\in[0,T],
  \]
  which implies uniqueness of $\xi$.
\end{proof}


\ifbozza\newpage\else\fi
\section{Proof of the main result}
\label{sec:pf}
Let $Y \in L^2(\Omega;L^2(0,T;H))$ be a progressively measurable
process, $X_0 \in L^2(\Omega,\cF_0,\P;H)$, and consider the
equation
\begin{equation}   \label{eq:XY}
dX(t) + AX(t)\,dt + \beta(X(t))\,dt \ni B(t,Y(t))\,dW(t),
\qquad X(0)=X_0.
\end{equation}
Since $B(\cdot,Y)$ is $U$-measurable, adapted, and belongs to
$L^2(\Omega;L^2(0,T;\cL^2(U,H)))$, the above equation is well-posed in
$\mathscr{J}$ by Proposition \ref{prop:add}, hence one can define a map
\begin{align*}
  \Gamma: L^2(\Omega;H) \times L^2(\Omega;L^2(0,T;H))
  &\longrightarrow L^2(\Omega;L^2(0,T;H))\times L^1(\Omega\times(0,T)\times D)\\
  (X_0,Y) &\longmapsto (X,\xi),
\end{align*}
where $(X,\xi)$ is the unique process in $\mathscr{J}$ solving
\eqref{eq:XY}.  Denoting the $L^2(\Omega; L^2(0,T; H))$-valued
component of $\Gamma$ by $\Gamma_1$ and the
$L^1(\Omega\times(0,T)\times D)$-valued component by $\Gamma_2$, we
are going to show that $Y \mapsto \Gamma_1(X_0,Y)$ is a (strict)
contraction of $L^2(\Omega;L^2(0,T;H))$, if endowed with a suitably
chosen equivalent norm. Let $X_i=\Gamma_1(X_{0i},Y_i)$, $i=1,2$, with
obvious meaning of the symbols. For any $\alpha \geq 0$, Proposition
\ref{prop:contra} yields
\begin{equation}
\label{eq:alfa}
\begin{split}
  &\norm[\big]{X_1-X_2}_{L^2(\Omega;L^\infty_\alpha(0,T;H)) \cap%
  L^2(\Omega;L^2_\alpha(0,T;V))}
  + \sqrt{\alpha} \norm[\big]{X_1-X_2}_{L^2(\Omega;L^2_\alpha(0,T;H))}\\
  &\hspace{5em} \lesssim \norm[\big]{X_{01}-X_{02}}_{L^2(\Omega;H)}
  + \norm[\big]{B(\cdot,Y_1)-B(\cdot,Y_2)}_{L^2(\Omega;L^2_\alpha(0,T;\cL^2(U,H)))},
\end{split}
\end{equation}
in particular, by the Lipschitz continuity of $B$,
\begin{align}
  \nonumber
  \norm[\big]{X_1-X_2}_{L^2(\Omega;L^2_\alpha(0,T;H))} &\lesssim
  \frac{1}{\sqrt{\alpha}} \norm[\big]{X_{01}-X_{02}}_{L^2(\Omega;H)}\\
  \nonumber
  &\quad + \frac{1}{\sqrt{\alpha}}
  \norm[\big]{B(\cdot,Y_1)-B(\cdot,Y_2)}_{L^2(\Omega;L^2_\alpha(0,T;\cL^2(U,H)))}\\
  \label{eq:alfetta}
  &\lesssim \frac{1}{\sqrt{\alpha}} \Bigl(
  \norm[\big]{X_{01}-X_{02}}_{L^2(\Omega;H)}
  + \norm[\big]{Y_1-Y_2}_{L^2(\Omega;L^2_\alpha(0,T;H))} \Bigr),
\end{align}
where the implicit constant does not depend on $\alpha$. In
particular, if $X_{01}=X_{02}$, choosing $\alpha$ large enough, one
has that, for any $X_0 \in L^2(\Omega,H)$, $Y \mapsto \Gamma_1(X_0,Y)$
is a contraction of $L^2(\Omega;L^2_\alpha(0,T;H))$. It follows by the
Banach fixed-point theorem that $\Gamma_1(X_0,\cdot)$ has a unique fixed
point $X$ therein, hence also in $L^2(\Omega;L^2(0,T;H))$ by equivalence
of norms. 
Setting $\xi:=\Gamma_2(X_0,X)$,
by definition of the map $\Gamma$, $(X,\xi)$
is a solution to \eqref{eq:0} and it belongs to
$\mathscr{J}$.

Let $X_{01}$, $X_{02}\in L^2(\Omega, \cF_0; H)$ and
$X_1$, $X_2$ be the unique fixed points of the maps $\Gamma_1(X_{0i},\cdot)$,
$i=1,2$, respectively, and $\xi_i:=\Gamma_2(X_{0i}, X_i)$, $i=1,2$.
Replacing $Y_i$ with $X_i=\Gamma_1(X_{0i},X_i)$, $i=1,2$, in
\eqref{eq:alfetta} yields
\[
\norm[\big]{X_1-X_2}_{L^2(\Omega;L^2_\alpha(0,T;H))} \leq
C_1 \norm[\big]{X_{01}-X_{02}}_{L^2(\Omega;H)}
+ C_2 \norm[\big]{X_1-X_2}_{L^2(\Omega;L^2_\alpha(0,T;H))},
\]
with $C_1>0$, $C_2\in \mathopen]0,1\mathclose[$, hence, by equivalence
of norms,
\[
\norm[\big]{X_1-X_2}_{L^2(\Omega;L^2(0,T;H))} \lesssim
\norm[\big]{X_{01}-X_{02}}_{L^2(\Omega;H)}.
\]
This implies, substituting $Y_i$ with $X_i=\Gamma(X_{0i},X_i)$,
$i=1,2$, in \eqref{eq:alfa}, with $\alpha=0$,
\begin{align*}
  &\norm[\big]{X_1-X_2}_{L^2(\Omega;L^\infty(0,T;H)) \cap%
  L^2(\Omega;L^2(0,T;V))}\\
  &\hspace{5em} \lesssim \norm[\big]{X_{01}-X_{02}}_{L^2(\Omega;H)}
  + \norm[\big]{B(\cdot,X_1)-B(\cdot,X_2)}_{L^2(\Omega;L^2(0,T;\cL^2(U,H)))}\\
  &\hspace{5em} \lesssim \norm[\big]{X_{01}-X_{02}}_{L^2(\Omega;H)}
  + \norm[\big]{X_1-X_2}_{L^2(\Omega;L^2(0,T;H))}\\
  &\hspace{5em} \lesssim \norm[\big]{X_{01}-X_{02}}_{L^2(\Omega;H)}.
\end{align*}
Choosing $\alpha=0$ and $X_{01}=X_{02}$, one gets immediately $X_1=X_2$,
hence also, by substitution,
\[
  \int_0^t(\xi_1(s)-\xi_2(s))\,ds = 0 \qquad\forall t\in[0,T],
\]
which implies uniqueness of $\xi$.



\ifbozza\newpage\else\fi
\bibliographystyle{amsplain}
\bibliography{ref}

\def\polhk#1{\setbox0=\hbox{#1}{\ooalign{\hidewidth
  \lower1.5ex\hbox{`}\hidewidth\crcr\unhbox0}}}
\providecommand{\bysame}{\leavevmode\hbox to3em{\hrulefill}\thinspace}
\providecommand{\MR}{\relax\ifhmode\unskip\space\fi MR }
\providecommand{\MRhref}[2]{%
  \href{http://www.ams.org/mathscinet-getitem?mr=#1}{#2}
}
\providecommand{\href}[2]{#2}
\begin{thebibliography}{10}

\bibitem{Kawa}
S.~Albeverio, H.~Kawabi, and M.~R{\"o}ckner, \emph{Strong uniqueness for both
  {D}irichlet operators and stochastic dynamics to {G}ibbs measures on a path
  space with exponential interactions}, J. Funct. Anal. \textbf{262} (2012),
  no.~2, 602--638. \MR{2854715}

\bibitem{ISEM:HK}
W.~Arendt, \emph{Heat kernels}, Lecture Notes of the 9th Internet Seminar on
  Evolution Equations, 2006, available at
  \url{https://www.uni-ulm.de/mawi/iaa/members/arendt/}.

\bibitem{AreBukh}
W.~Arendt and A.~V. Bukhvalov, \emph{Integral representations of resolvents and
  semigroups}, Forum Math. \textbf{6} (1994), no.~1, 111--135. \MR{1253180}

\bibitem{ISEM18}
W.~Arendt, R.~Chill, C.~Seifert, D.~Vogt, and J.~Voigt, \emph{Form methods for
  evolution equations and applications}, Lecture Notes of the 18th Internet
  Seminar on Evolution Equations, 2015, available at
  \url{https://www.mat.tuhh.de/isem18/Phase_1:_The_lectures}.

\bibitem{barbu}
V.~Barbu, \emph{Analysis and control of nonlinear infinite-dimensional
  systems}, Academic Press Inc., Boston, MA, 1993. \MR{MR1195128 (93j:49002)}

\bibitem{Barbu:Lincei}
\bysame, \emph{Existence for semilinear parabolic stochastic equations}, Atti
  Accad. Naz. Lincei Cl. Sci. Fis. Mat. Natur. Rend. Lincei (9) Mat. Appl.
  \textbf{21} (2010), no.~4, 397--403. \MR{2746091 (2012d:35424)}

\bibitem{BDPR-porous}
V.~Barbu, G.~Da~Prato, and M.~R{\"o}ckner, \emph{Existence of strong solutions
  for stochastic porous media equation under general monotonicity conditions},
  Ann. Probab. \textbf{37} (2009), no.~2, 428--452. \MR{MR2510012}

\bibitem{cm:IDAQP09}
V.~Barbu and C.~Marinelli, \emph{Strong solutions for stochastic porous media
  equations with jumps}, Infin. Dimens. Anal. Quantum Probab. Relat. Top.
  \textbf{12} (2009), no.~3, 413--426. \MR{2572464}

\bibitem{Bend:Nash}
A.~Bendikov and P.~Maheux, \emph{Nash type inequalities for fractional powers
  of non-negative self-adjoint operators}, Trans. Amer. Math. Soc. \textbf{359}
  (2007), no.~7, 3085--3097 (electronic). \MR{2299447}

\bibitem{BoCro}
L.~Boccardo and G.~Croce, \emph{Elliptic partial differential equations}, De
  Gruyter, Berlin, 2014. \MR{3154599}

\bibitem{Bbk:EVT}
N.~Bourbaki, \emph{Espaces vectoriels topologiques. {C}hapitres 1 \`a 5}, new
  ed., Masson, Paris, 1981. \MR{633754}

\bibitem{Bre-mm}
H.~Br{\'e}zis, \emph{Monotonicity methods in {H}ilbert spaces and some
  applications to nonlinear partial differential equations}, Contributions to
  nonlinear functional analysis (Proc. Sympos., Math. Res. Center, Univ.
  Wisconsin, Madison, Wis., 1971), Academic Press, New York, 1971,
  pp.~101--156. \MR{MR0394323 (52 \#15126)}

\bibitem{Bmax}
\bysame, \emph{Op\'erateurs maximaux monotones et semi-groupes de contractions
  dans les espaces de {H}ilbert}, North-Holland Publishing Co., Amsterdam,
  1973. \MR{MR0348562 (50 \#1060)}

\bibitem{beam}
Z.~Brze\'zniak, B.~Maslowski, and J.~Seidler, \emph{Stochastic nonlinear beam
  equations}, Probab. Theory Related Fields \textbf{132} (2005), no.~1,
  119--149. \MR{2136869}

\bibitem{cerrai03}
S.~Cerrai, \emph{Stochastic reaction-diffusion systems with multiplicative
  noise and non-{L}ipschitz reaction term}, Probab. Theory Related Fields
  \textbf{125} (2003), no.~2, 271--304. \MR{1961346 (2004a:60117)}

\bibitem{DP-K}
G.~Da~Prato, \emph{Kolmogorov equations for stochastic {PDE}s}, Birkh\"auser
  Verlag, Basel, 2004. \MR{MR2111320 (2005m:60002)}

\bibitem{DauLio:2}
R.~Dautray and J.-L. Lions, \emph{Mathematical analysis and numerical methods
  for science and technology. {V}ol. 2}, Springer-Verlag, Berlin, 1988.
  \MR{969367}

\bibitem{Dav-HK}
E.~B. Davies, \emph{Heat kernels and spectral theory}, Cambridge University
  Press, Cambridge, 1990. \MR{MR1103113 (92a:35035)}

\bibitem{EnNa}
K.-J. Engel and R.~Nagel, \emph{One-parameter semigroups for linear evolution
  equations}, Springer-Verlag, New York, 2000. \MR{MR1721989 (2000i:47075)}

\bibitem{Gentil:Levy}
I.~Gentil and C.~Imbert, \emph{The {L}\'evy-{F}okker-{P}lanck equation:
  {$\Phi$}-entropies and convergence to equilibrium}, Asymptot. Anal.
  \textbf{59} (2008), no.~3-4, 125--138. \MR{2450356}

\bibitem{Haa07}
M.~Haase, \emph{Convexity inequalities for positive operators}, Positivity
  \textbf{11} (2007), no.~1, 57--68. \MR{2297322 (2008d:39034)}

\bibitem{lema}
J.-B. Hiriart-Urruty and C.~Lemar{\'e}chal, \emph{Fundamentals of convex
  analysis}, Springer-Verlag, Berlin, 2001. \MR{1865628 (2002i:90002)}

\bibitem{Kato}
T.~Kato, \emph{Perturbation theory for linear operators}, Springer-Verlag,
  Berlin, 1995, Reprint of the 1980 edition. \MR{1335452}

\bibitem{KR-spde}
N.~V. Krylov and B.~L. Rozovski{\u\i}, \emph{Stochastic evolution equations},
  Current problems in mathematics, Vol. 14 (Russian), Akad. Nauk SSSR,
  Vsesoyuz. Inst. Nauchn. i Tekhn. Informatsii, Moscow, 1979, pp.~71--147, 256.
  \MR{MR570795 (81m:60116)}

\bibitem{KvN2}
M.~Kunze and J.~van Neerven, \emph{Continuous dependence on the coefficients
  and global existence for stochastic reaction diffusion equations}, J.
  Differential Equations \textbf{253} (2012), no.~3, 1036--1068. \MR{2922662}

\bibitem{cm:AIHP14}
Se. Kusuoka and C.~Marinelli, \emph{On smoothing properties of transition
  semigroups associated to a class of {SDE}s with jumps}, Ann. Inst. Henri
  Poincar\'e Probab. Stat. \textbf{50} (2014), no.~4, 1347--1370. \MR{3269997}

\bibitem{LiMa1}
J.-L. Lions and E.~Magenes, \emph{Probl\`emes aux limites non homog\`enes et
  applications. {V}ol. 1}, Dunod, Paris, 1968. \MR{0247243}

\bibitem{LiuRo}
Wei Liu and M.~R{\"o}ckner, \emph{Stochastic partial differential equations: an
  introduction}, Springer, Cham, 2015. \MR{3410409}

\bibitem{MR:DF}
Zhi~Ming Ma and M.~R{\"o}ckner, \emph{Introduction to the theory of
  (nonsymmetric) {D}irichlet forms}, Springer-Verlag, Berlin, 1992.
  \MR{1214375}

\bibitem{cm:rd}
C.~Marinelli, \emph{On well-posedness of semilinear stochastic evolution
  equations on {$L_p$} spaces}, arXiv:1512.04323.

\bibitem{cm:SIMA12}
C.~Marinelli and Ll. Quer-Sardanyons, \emph{Existence of weak solutions for a
  class of semilinear stochastic wave equations}, SIAM J. Math. Anal.
  \textbf{44} (2012), no.~2, 906--925. \MR{2914254}

\bibitem{Ouhabaz}
E.~M. Ouhabaz, \emph{Analysis of heat equations on domains}, Princeton
  University Press, Princeton, NJ, 2005. \MR{2124040}

\bibitem{Pard}
\'E. Pardoux, \emph{Equations aux deriv\'ees partielles stochastiques
  nonlin\'eaires monotones}, Ph.D. thesis, Universit\'e Paris XI, 1975.

\bibitem{pardoux-rascanu}
\'E. Pardoux and A.~R\u{a}\c{s}canu, \emph{Stochastic differential equations,
  backward {SDE}s, partial differential equations}, Springer, Cham, 2014.
  \MR{3308895}

\bibitem{Pazy}
A.~Pazy, \emph{Semigroups of linear operators and applications to partial
  differential equations}, Springer-Verlag, New York, 1983. \MR{85g:47061}

\bibitem{Simon}
J.~Simon, \emph{Compact sets in the space {$L^p(0,T;B)$}}, Ann. Mat. Pura Appl.
  (4) \textbf{146} (1987), 65--96. \MR{916688 (89c:46055)}

\bibitem{Strauss}
W.~A. Strauss, \emph{On continuity of functions with values in various {B}anach
  spaces}, Pacific J. Math. \textbf{19} (1966), 543--551. \MR{0205121 (34
  \#4956)}

\bibitem{vNVW}
J.~van Neerven, M.~C. Veraar, and L.~Weis, \emph{Stochastic evolution equations
  in {UMD} {B}anach spaces}, J. Funct. Anal. \textbf{255} (2008), no.~4,
  940--993. \MR{2433958 (2009h:35465)}

\bibitem{VSC}
N.~Th. Varopoulos, L.~Saloff-Coste, and T.~Coulhon, \emph{Analysis and geometry
  on groups}, Cambridge University Press, Cambridge, 1992. \MR{1218884}

\end{thebibliography}

\end{document}